\documentclass[a4paper,11pt]{article}

\usepackage[english]{babel}
\usepackage[latin1]{inputenc}
\usepackage{lmodern}
\usepackage{amsfonts}
\usepackage{bm}
\usepackage{amssymb}
\usepackage{latexsym}
\usepackage{amsmath}
\usepackage{amscd}
\usepackage{amsthm}
\usepackage{pifont}
\usepackage{fancyhdr}
\usepackage{epsfig}
\usepackage{alltt}
\usepackage{hyperref}
\usepackage{multimedia}
\usepackage[T1]{fontenc}
\usepackage{psfrag}
\usepackage{url}
\usepackage{graphicx}
\usepackage{psfrag}
\usepackage{ulem}

\newtheorem{theorem}{Theorem}
\newtheorem{definition}{Definition}
\newtheorem{proposition}[theorem]{Proposition}
\newtheorem{propri}[theorem]{Property}
\newtheorem{lemma}[theorem]{Lemma}

\newtheorem{prop}[theorem]{Proposition}

\def\arc{\mathrm{a}}
\def\Rect{\mathrm{Rect}}

\def\A{\mathbf{A}}

\def\eps{\varepsilon}
\def\II{\mbox{\rm \large 1\hskip -0.353em 1}}
\newcommand{\EE}{\mathop{\hbox{\rm I\kern-0.17em E}}\nolimits}
\newcommand{\PP}{\mathop{\hbox{\rm I\kern-0.17em P}}\nolimits}
\newcommand{\NN}{{\mathcal N}}
\newcommand{\TT}{{\mathcal T}}
\newcommand{\FF}{{\mathcal F}}
\newcommand{\CC}{{\mathcal C}}
\newcommand{\g}{{\gamma}}

\newcommand{\N}{{\mathbb N}}
\newcommand{\Z}{{\mathbb Z}}

\newcommand{\R}{{\mathbb R}}

\newcommand{\ind}{{\bf 1}}

\def\i{\mathbf{i}}

\def\etal{\textit{et al.} }

\title{Sublinearity of the number of semi-infinite branches for geometric random trees}

\author{David Coupier}

\date{\today}

\begin{document}
\maketitle

\begin{abstract}
The present paper addresses the following question: for a geometric random tree in $\R^{2}$, how many semi-infinite branches cross the circle $\CC_{r}$ centered at the origin and with a large radius $r$? We develop a method ensuring that the expectation of the number $\chi_{r}$ of these semi-infinite branches is $o(r)$. The result follows from the fact that, far from the origin, the distribution of the tree is close to that of an appropriate directed forest which lacks bi-infinite paths. In order to illustrate its robustness, the method is applied to three different models: the Radial Poisson Tree (RPT), the Euclidean First-Passage Percolation (FPP) Tree and the Directed Last-Passage Percolation (LPP) Tree. Moreover, using a coalescence time estimate for the directed forest approximating the RPT, we show that for the RPT $\chi_{r}$ is $o(r^{1-\eta})$, for any $0<\eta<1/4$, almost surely and in expectation.
\end{abstract}



\section{Introduction}


The present paper focuses on geometric random trees embedded in $\R^{2}$ and on their semi-infinite paths. When each vertex of a given geometric random tree $\TT$ built on a countable vertex set has finite degree then $\TT$ automatically admits at least one semi-infinite path. Excepting this elementary result, describing the semi-infinite paths of $\TT$ (their number, their directions etc.) is nontrivial. An important step was taken by Howard and Newman in \cite{HN01}. They develop an efficient method (Proposition 2.8) ensuring that, under certain hypotheses, a geometric random tree $\TT$ satisfies the two following statements : with probability $1$, every semi-infinite path of $\TT$ has an asymptotic direction $[S1]$ and for every direction $\theta\in[0;2\pi)$, there is at least one semi-infinite path of $\TT$ with asymptotic direction $\theta$ $[S2]$. As a consequence, the number $\chi_{r}$ of semi-infinite paths of $\TT$ crossing the circle $\CC_{r}$ tends to infinity as $r\to\infty$. Thenceforth, a natural question already mentioned in the seminal 1965 article of Hammersley and Welsh \cite{HaWe} as the \textit{highways and byways problem} concerns the growth rate of $\chi_{r}$. To our knowledge, this problem has not been studied until now.

In this paper, we state bounds for this rate in three particular models: the Radial Poisson Tree (RPT), the Euclidean First-Passage Percolation (FPP) Tree and the Directed Last-Passage Percolation (LPP) Tree.\\

In the last 15 years, many geometric random trees satisfying $[S1]$ and $[S2]$ are appeared in the literature. Among these trees, two different classes can be distinguished. The first one is that of \textit{greedy} trees (as the RPT). The graph structure of a greedy tree results from local rules. Its paths are obtained through greedy algorithms, e.g. each vertex is linked to the closest point inside a given set. On the contrary, \textit{optimized} trees are built from global rules as first or last-passage percolation procedure. This is the case of the Euclidean FPP Tree and the Directed LPP Tree. Besides, for these trees, one refers to geodesics instead of branches or paths.

Our first main result says that the mean number of semi-infinite paths is asymptotically sublinear, i.e.
\begin{equation}
\label{sublinChir}
\EE \chi_{r} = o(r) ~,
\end{equation}
for many examples of greedy and optimized trees. This result means that among all the edges crossing the circle $\CC_{r}$, whose mean number is of order $r$, a very few number of them belong to semi-infinite paths.
 
Let us first give some examples of greedy trees studied in the literature. From now on, $\NN$ denotes a homogeneous Poisson Point Process (PPP) in $\R^{2}$ with intensity $1$. The \textit{Radial Spanning Tree} (RST) has been introduced by Baccelli and Bordenave in \cite{BB} to modelize communication networks. This tree, rooted at the origin $O$ and whose vertex set is $\NN\cup\{O\}$, is defined as follows: each vertex $X\in\NN$ is linked to its closest vertex among $(\NN\cup\{O\})\cap D(O,|X|)$. A second example is given by Bonichon and Marckert in \cite{BM}. The authors study the \textit{Navigation Tree} in which each vertex is linked to the closest one in a given sector-- with angle $\theta$ --oriented towards the origin $O$ (see Section 1.2.2). This tree satisfies $[S1]$ and $[S2]$ whenever $\theta$ is not too large (Theorem 5). In Section \ref{sect:RPT} of this paper, a third example of greedy tree is introduced, called the \textit{Radial Poisson Tree} (RPT). For any vertex $X\in\NN$, let $\textrm{Cyl}(X,\rho)$ be the set of points of the disk $D(O,|X|)$ whose distance to the segment $[O;X]$ is smaller than a given parameter $\rho>0$. Then, in the RPT, $X$ is linked to the element of $\textrm{Cyl}(X,\rho)\cap(\NN\cup\{O\})$ having the largest Euclidean norm. The main motivation to study this model comes from the fact it is closely related to a directed forest studied by Ferrari and his coauthors in \cite{FLT,FFW}. In Theorem \ref{theo:RPTstraight} of Section \ref{sect:RPTstraight}, we prove that statements $[S1]$ and $[S2]$ hold for the RPT.

Our first example of optimized tree is called the \textit{Euclidean FPP Tree} and has been introduced by Howard and Newman in \cite{HN01}. In this tree, the geodesic joining each vertex $X$ of a PPP $\NN$ to the root $X^{O}$, which is the closest point of $\NN$ to the origin $O$, is defined as the path $X_{1}=X^{O},\ldots,X_{n}=X$ minimizing the weight $\sum_{i=1}^{n-1} |X_{i}-X_{i+1}|^{\alpha}$, where $\alpha>0$ is a given parameter. A second example is given by Pimentel in \cite{Pimentel}. First, the author associates i.i.d. nonnegative random variables to the edges of the Delaunay triangulation built from the PPP $\NN$. Thus, he links each vertex $Y$ of the triangulation to a selected root by a FPP procedure. Our third example of optimized tree is slightly different from the previous ones since its vertex set is the deterministic grid $\N^{2}$. The \textit{Directed LPP Tree} is obtained by a LPP procedure from i.i.d. random weights associated to the vertices of $\N^{2}$. Under suitable hypotheses, this tree satisfies $[S1]$ and $[S2]$ (see Georgiou \etal \cite{GeRaSe}).

The sublinear result (\ref{sublinChir}) has been already proved for the RST. But its proof is scattered in several works \cite{BB,BCT,CT} and its robust and universal character has not been sufficiently highlighted. We claim that our method allows to show that (\ref{sublinChir}) holds for all the trees mentioned above and we explicitly prove it for the RPT (Theorem \ref{theo:sublinRPT}), the Euclidean FPP Tree (Theorem \ref{theo:sublinEuclFPPtree}) and the Directed LPP Tree (Theorem \ref{theo:sublinLPPtree}).

When the law of the geometric random tree is isotropic, the limit (\ref{sublinChir}) follows from
\begin{equation}
\label{sublinDirectional}
\EE \chi_{r}(\theta,2\pi) = o(1) ~,
\end{equation}
where $\chi_{r}(\theta,2\pi)$ is the number of semi-infinite paths of the considered tree $\TT$ crossing the arc of $\CC_{r}$ centered at $r e^{\i\theta}$ and with length $2\pi$. The proof of (\ref{sublinDirectional}) is mainly based on two ingredients. First, we approximate the distribution of $\TT$, locally and around the point $re^{\i\theta}$, by a directed forest $\FF$ with direction $-e^{\i\theta}$. The proof of this approximation result and the definition of the approximating directed forest $\FF$ are completely different according to the nature (greedy or optimized) of $\TT$. Roughly speaking, in the greedy case, the approximating directed forest $\FF$ is obtained from $\TT$ by sending the target, i.e. the root of $\TT$, to infinity in the direction $-e^{\i\theta}$. In particular, the directed forest approximating the RPT is the collection of coalescing one-dimensional random walks with uniform jumps in a bounded interval with radius $\rho$ (see \cite{FLT}). When $\TT$ is optimized, the approximating directed forest is given by the collection of semi-infinite geodesics having the same direction $-e^{\i\theta}$ and starting at all the vertices. The existence of these semi-infinite geodesics is ensured by $[S2]$. Their uniqueness is stated in Proposition \ref{prop:directiondeterm}. The second ingredient is the absence of bi-infinite path in the approximating directed forest. Actually, this is the only part of our method where the dimension two is used, and even required for optimized trees. Let us also add that our method applies even if the limit shape of the considered model is unknown.

In Theorems \ref{theo:sublinRPT}, \ref{theo:sublinEuclFPPtree} and \ref{theo:sublinLPPtree}, it is also established that a.s. $\chi_{r}(\theta,2\pi)$ does not tend to $0$. This is due to the absence of double semi-infinite paths with the same deterministic direction.\\

Our second main result (also in Theorem \ref{theo:sublinRPT}) is a substantial improvement of (\ref{sublinChir}) and (\ref{sublinDirectional}) in the case of the RPT. We prove that as $r\to\infty$
\begin{equation}
\label{sublinChirImprov}
\EE \chi_{r}(\theta,r^{\eta}) = o(1) ~,
\end{equation}
for any $0<\eta<1/4$, and then by isotropy $\EE\chi_{r}$ is $o(r^{1-\eta})$. As for the proof of (\ref{sublinDirectional}), the one of (\ref{sublinChirImprov}) also uses the approximation result by a suitable directed forest $\FF$ but this times in a non local way (see Lemma \ref{lem:RPT-Claim2}). Indeed, unlike $\chi_{r}(\theta,2\pi)$, the arc involved in $\chi_{r}(\theta,r^{\eta})$ has a size growing with $r$. Moreover, accurate estimates on fluctuations of paths of $\FF$ and on the coalescence time of two given paths are needed. It is worth pointing out here that a deep link seems to exist between the rate at which $\chi_{r}$ tends to $0$ and the rate at which semi-infinite paths merge in the approximating directed forest.

Furthermore, we deduce from (\ref{sublinChirImprov}) an almost sure convergence: with probability $1$, the ratio $\chi_{r}/r^{1-\eta}$ tends to $0$ as $r\to\infty$ for any $0<\eta<1/4$. This result is based on the fact that two semi-infinite paths of the RPT, far from each other, are independent with high probability. This argument does not hold in the FPP/LPP context.

Let us finally remark that the convergences in mean obtained in this paper for the RPT (i.e. limits (\ref{sublinRPT}) of Theorem \ref{theo:sublinRPT}) do not require the statements $[S1]$ and $[S2]$. However, this is not the case of the almost sure results (i.e. limits (\ref{CVasRPT}) and (\ref{NoCVasRPT}) of Theorem \ref{theo:sublinRPT}) and this is the reason why we prove in Section \ref{sect:RPTstraight} that the RPT satisfies these two statements.\\

Our paper is organized as follows. In Section \ref{sect:ModelsResults}, the RPT, the Euclidean FPP Tree and the Directed LPP Tree are introduced, and the sublinearity results (Theorems \ref{theo:sublinRPT}, \ref{theo:sublinEuclFPPtree} and \ref{theo:sublinLPPtree}) are stated. The general scheme of the proof of (\ref{sublinDirectional}) is developped in Section \ref{sect:SketchSublin} but the reader must refer to Sections \ref{sect:proofEuclFPPtree} and \ref{sect:proofLPP} respectively for details about the Euclidean FPP Tree and the Directed LPP Tree. The last three sections concern the RPT. The proof of (\ref{sublinChirImprov}) is devoted to Section \ref{sect:proofRPT}. The almost sure convergence of $\chi_{r}/r^{1-\eta}$ to $0$ is given in Section \ref{sect:asRPT}. Finally, in Section \ref{sect:RPTstraight}, we prove that the RPT satisfies statements $[S1]$ and $[S2]$ (Theorem \ref{theo:RPTstraight}).

\section{Models and sublinearity results}
\label{sect:ModelsResults}

Let $O$ be the origin of $\R^{2}$ which is endowed with the Euclidean norm $|\cdot|$. We denote by $D(x,r)$ the open Euclidean disk with center $x$ and radius $r$, and by $\CC(x,r)$ the corresponding disk. We merely set $\CC_{r}$ instead of $\CC(O,r)$. Let $\arc_{r}(\theta,c)$ be the arc of $\CC_{r}$ centered at $re^{\i\theta}$ and with length $c>0$. The Euclidean scalar product is denoted by $\langle\cdot,\cdot\rangle$. Throughout the paper, the real plan $\R^{2}$ and the set of complex numbers $\mathbb{C}$ are identified. Hence, according to the context, a point $X$ will be described by its cartesian coordinates $(X(1),X(2))$ or its Euclidean norm $|X|$ and its argument $\mbox{arg}(X)$ which is the unique (when $X\not= O$) real number $\theta$ in $[0;2\pi)$ such that $X=|X|e^{\i\theta}$. We denote by $[X;Y]$ the segment joining $X,Y\in\R^{2}$ and by $(X;Y)$ the corresponding open segment.

All the trees and forests considered in the sequel are graphs with out-degree $1$ (except for their roots). Hence, they are naturally directed. The outgoing vertex $Y$ of any edge $(X,Y)$ will be denoted by $A(X)$ and called the \textit{ancestor} of $X$. We will also say that $X$ is a \textit{child} of $A(X)$. Moreover, it will be convenient to keep the same notation for these trees and forests that they are considered as random graphs with directed edges $(X,A(X))$ or as subsets of $\R^{2}$ made up of segments $[X;A(X)]$.

A sequence $(X_{n})_{n\in\N}$ of vertices defines a semi-infinite path (resp. a bi-infinite path) of a geometric random graph if for any $n\in\N$ (resp. $n\in\Z$), $A(X_{n+1})=X_{n}$. A semi-infinite path $(X_{n})_{n\in\N}$ admits $\theta\in[0;2\pi)$ as asymptotic direction if
$$
\lim_{n\to\infty} \frac{X_{n}}{|X_{n}|} = e^{\i\theta} ~.
$$

The number of semi-infinite paths at level $r$, i.e. crossing the circle $\CC_{r}$, will be denoted by $\chi_{r}$. This notion should be specified according to the context.

In the sequel, $\NN$ denotes a homogeneous PPP in $\R^{2}$ with intensity $1$. The number of Poisson points in a given measurable set $\Lambda$ is $\NN(\Lambda)$.

\subsection{The Radial Poisson Tree}
\label{sect:RPT}

Let $\rho>0$ be a positive real number. Considering $\NN\cap D(O,\rho)^{c}$ instead of $\NN$, we can assume for this section that $\NN$ has no point in $D(O,\rho)$. The \textit{Radial Poisson Tree} (RPT) $\TT_{\rho}$ is a directed graph whose vertex set is given by $\NN\cup\{O\}$. Let us define the ancestor $A(X)$ of any vertex $X\in\NN$ as follows. First we set
$$
\textrm{Cyl}(X,\rho) = \left( [O ; X] \oplus D(O,\rho) \right) \cap D(O,|X|) ~,
$$ 
where $\oplus$ denotes the Minkowski sum. If $\textrm{Cyl}(X,\rho)\cap\NN$ is empty then $A(X)=O$. Otherwise, $A(X)$ is the element of $\textrm{Cyl}(X,\rho)\cap\NN$ having the largest Euclidean norm:
\begin{equation}
\label{argmaxRPT}
A(X) = \arg\!\max \left\{ |Y| , \; Y \in \textrm{Cyl}(X,\rho)\cap\NN \right\} ~.
\end{equation}
Hence, the set $\textrm{Cyl}(X,\rho)^{\ast}=\textrm{Cyl}(X,\rho)\setminus D(O,|A(X)|)$ avoids the PPP $\NN$. See Figure \ref{fig:cylindre}. Let us note that the definition of ancestor $A(X)$ can be extended to any $X\in\R^{2}$.

\begin{figure}[!ht]
\begin{center}
\psfrag{a}{\small{$\CC_{\rho}$}}
\psfrag{b}{\small{$A(X)$}}
\psfrag{c}{\small{$X$}}
\psfrag{d}{\small{$\CC_{|X|}$}}
\includegraphics[width=7.5cm,height=5cm]{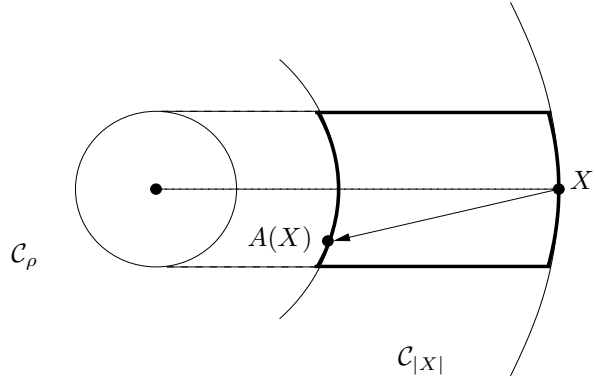}
\end{center}
\caption{\label{fig:cylindre} This picture represents the ancestor $A(X)\in\NN$ of $X=(x,0)$. The bold lines delimit the cylinder $\textrm{Cyl}(X,\rho)^{\ast}$. Remark that the distance between $A(X)$ and the segment $[O ;X]$ is smaller than $\rho$ in the RPT $\TT_{\rho}$ whereas it is unbounded in the Radial Spanning Tree (see \cite{BB}).}
\end{figure}

This construction ensures the a.s. uniqueness of the ancestor $A(X)$ of any $X\in\R^{2}$. This means that the RPT has no loop. Furthermore, $A(X)$ is closer than $X$ to the origin. Since the PPP $\NN$ is locally finite, then any $X\in\R^{2}$ is linked to the origin by a finite number of edges.

Here are some basic properties of the RPT $\TT_{\rho}$.

\begin{propri}
\label{propri:RPT}
The Radial Poisson Tree $\TT_{\rho}$ satisfies the following non-crossing property: for any vertices $X,X'\in\NN$ with $X\not= X'$, $(X ; A(X))\cap (X' ; A(X')) = \emptyset$. Moreover, the number of children of any $X\in\NN\cup\{O\}$ is a.s. finite but unbounded.
\end{propri}

\begin{proof}
Let $X,X'\in\NN$ with $X\not= X'$. By symmetry, we assume that a.s. $|X|>|X'|$. Actually, we can focus on the case where $|X'|>|A(X)|$. Otherwise, $|X|>|A(X)|\geq |X'|\geq\rho$ since the PPP $\NN$ avoids the disk $D(O,\rho)$. Henceforth the interval $(X ; A(X))$ is outside the disk $D(O,|X'|)$ whereas $(X' ; A(X'))$ is inside. They cannot overlap. Hence, let us assume that a.s. $\min\{|X|,|X'|\}>\max\{|A(X)|,|A(X')|\}$. If the ancestors $A(X)$ and $A(X')$ belong to $\textrm{Cyl}(X,\rho)\cap\textrm{Cyl}(X',\rho)$ then they necessarily are equal. Otherwise $A(X)$ belongs to $\textrm{Cyl}(X,\rho)\setminus\textrm{Cyl}(X',\rho)$ or $A(X')$ belongs to $\textrm{Cyl}(X',\rho)\setminus\textrm{Cyl}(X,\rho)$. In both cases, the sets $(X ; A(X))$ and $(X' ; A(X'))$ cannot overlap. [This is the reason of the hypothesis $\NN\cap D(O,\rho)=\emptyset$ which ensures that pathes of $\TT_{\rho}$ do not cross.]

About the second statement, we only treat the case of the origin $O$. These are the same arguments for any $X\in\NN$. Let $K$ be the number of children of $O$. By the Campbell's formula,
\begin{eqnarray*}
\EE K & = & \sum_{n\in\N} \; \EE \left\lbrack \sum_{{X\in\NN}\atop n\leq |X|<n+1} \ind_{A(X)=O} \right\rbrack \\
& = & \sum_{n\in\N} 2\pi \int_{n}^{n+1} \PP(A((x,0))=O) x \, dx \\
& \leq & \sum_{n\in\N} 2 \pi (n+1) \PP(A((n,0))=O) \\
& \leq & \sum_{n\in\N} 2 \pi (n+1) e^{-\rho(n-\rho)} < \infty ~.
\end{eqnarray*}
Then, the random variable $K$ is a.s. finite.

Let $R>\rho$ be a (large) real number. Consider a deterministic sequence of $k$ points $u_{1},\ldots,u_{k}$ on the circle $\CC_{R}$ such that $|u_{i}-u_{i+1}|=2\rho$ for $i=1,\ldots,k$. Such a sequence exists when
$$
k = \left\lfloor \frac{\pi}{\arcsin(\rho/R)} \right\rfloor ~,
$$
and in this case, $|u_{k}-u_{1}|\geq 2\rho$. Recall that $\lfloor\cdot\rfloor$ denotes the integer part. Let $0<\eps<\rho/2$. On the event ``each disk $D(u_i,\eps)$ contains exactly one point of $\NN$ and these are the only points of $\NN$ in the (large) disk $D(O,R+\eps)$'' which occurs with a positive probability, the number $K$ of children of $O$ is at least equal to $k$. Finally, the integer $k=k(R)$ tends to infinity with $R$. 
\end{proof}

\begin{figure}[!ht]
\begin{center}
\begin{tabular}{cc}
\includegraphics[width=7.5cm,height=7.5cm]{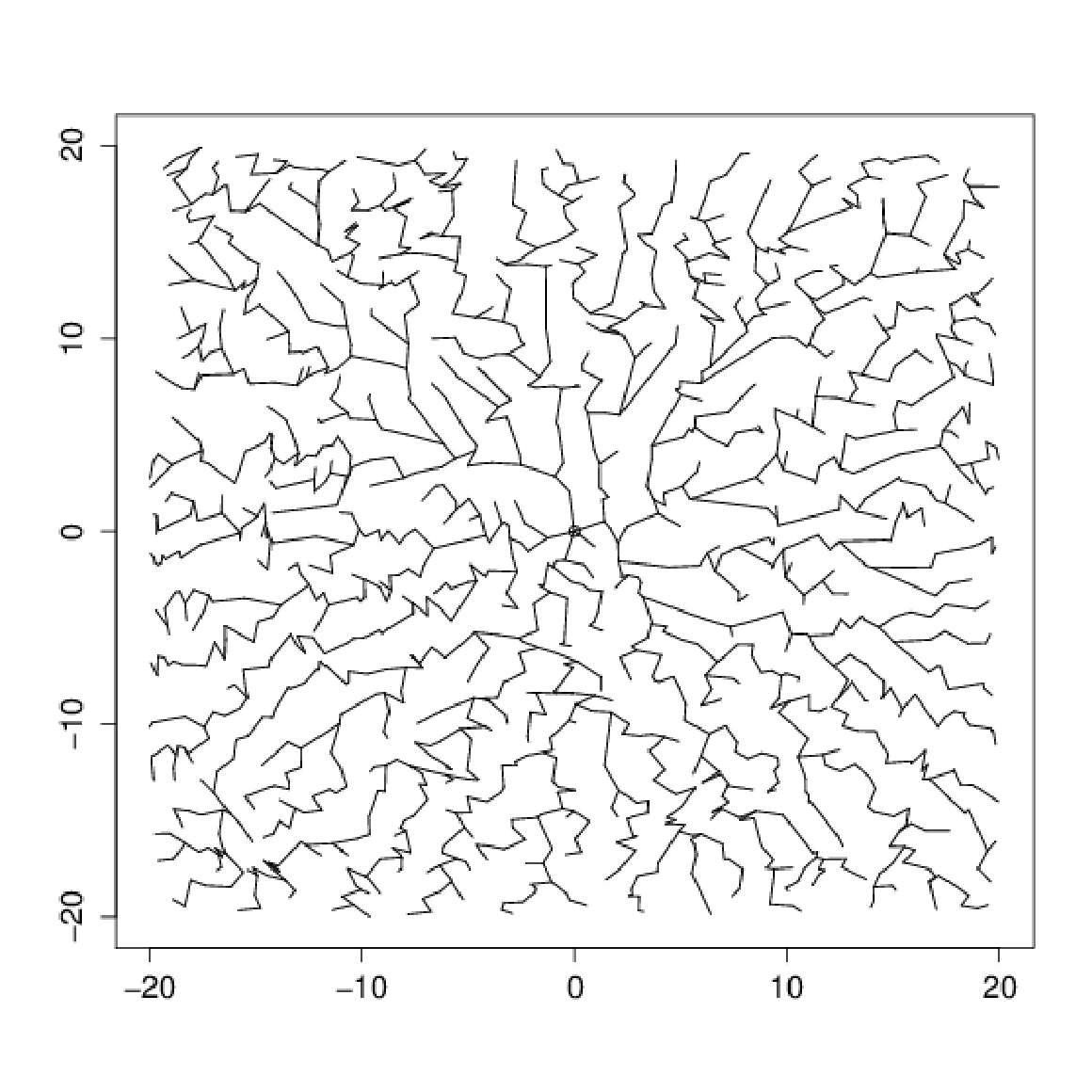} & \includegraphics[width=7.5cm,height=7.5cm]{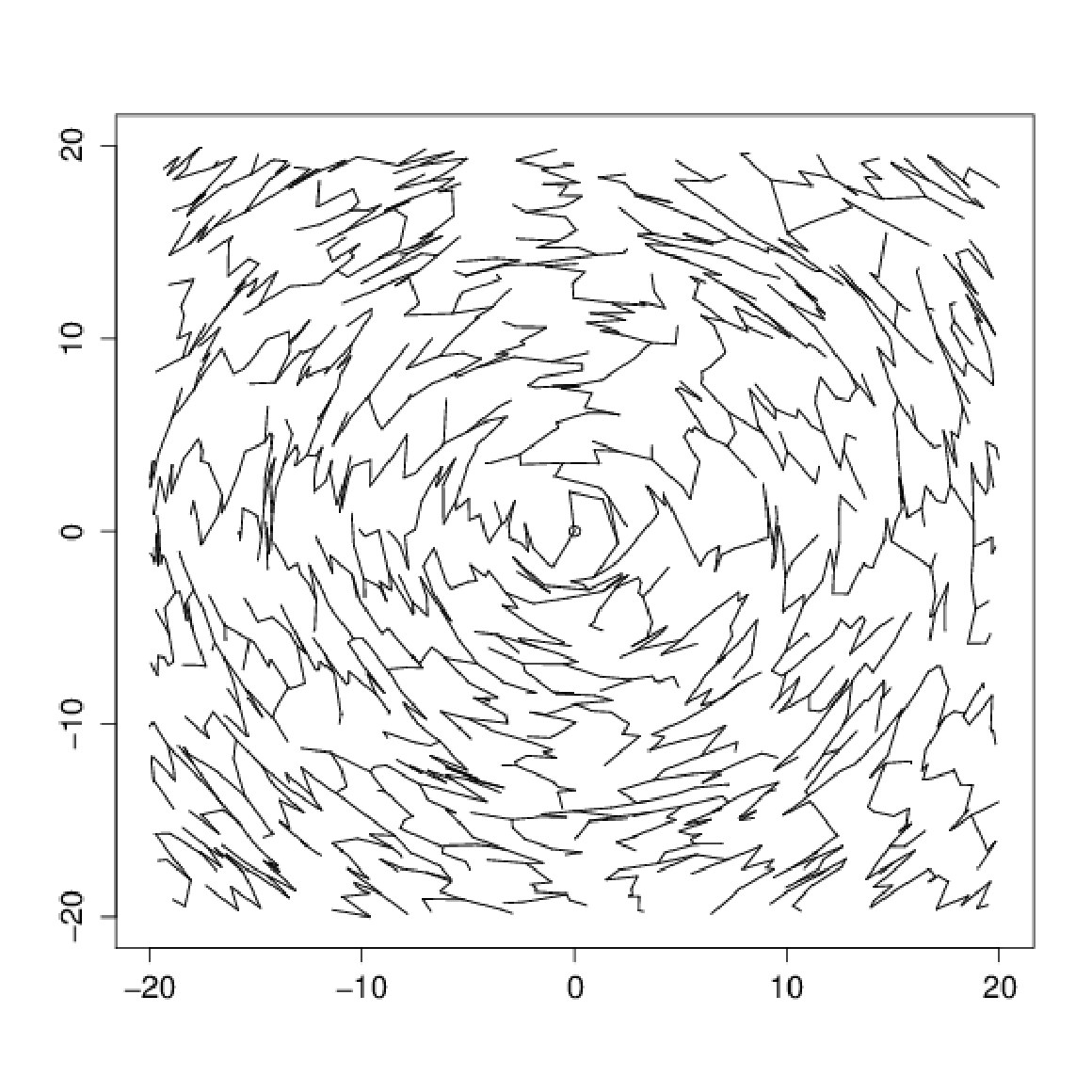}
\end{tabular}
\end{center}
\caption{Built on the same realization of the PPP $\NN$, the Radial Poisson Tree with $\rho=1$ (to the left) and $\rho=2$ (to the right).}
\end{figure}

Theorem \ref{theo:RPTstraight} of Section \ref{sect:RPTstraight} says that the RPT is \textit{straight}. Roughly speaking, this means that the subtrees of $\TT$ are becoming thinner when their roots are far away from the one of $\TT$. This notion has been introduced in Section 2.3 of \cite{HN01} to prove that any semi-infinite path has an asymptotic direction and in each direction there is a semi-infinite path (see Proposition 2.8 of \cite{HN01}). This is case of the RPT:

\begin{prop}
\label{prop:RPT-S12}
The Radial Poisson Tree $\TT_{\rho}$ a.s. satisfies statements $[S1]$ and $[S2]$.
\end{prop}

The random integer $\chi_{r}$ denotes the number of intersection points of the circle $\CC_{r}$ with the semi-infinite paths of the RPT. Proposition \ref{prop:RPT-S12} implies that $\chi_{r}$ a.s. converges to infinity as $r\to\infty$. Thus, consider two real numbers $\theta\in[0 ; 2\pi)$ and $c>0$. We denote by $\chi_{r}(\theta,c)$ the number of semi-infinite paths of the RPT crossing the arc $\arc_{r}(\theta,c)$ of $\CC_{r}$. Here is the sublinearity result satisfied by the RPT:

\begin{theorem}
\label{theo:sublinRPT}
Let $\theta\in[0 ; 2\pi)$ and $0<\eta<1/4$. Then,
\begin{equation}
\label{sublinRPT}
\lim_{r\to\infty} \EE \frac{\chi_{r}}{r^{1-\eta}} = 0 \; \mbox{ and } \; \lim_{r\to\infty} \EE \chi_{r}(\theta,r^{\eta}) = 0 ~.
\end{equation}
Furthermore,
\begin{equation}
\label{CVasRPT}
\lim_{r\to\infty} \frac{\chi_{r}}{r^{1-\eta}} = 0 \mbox{ a.s.}
\end{equation}
whereas, for any $c>0$, the sequence $(\chi_{r}(\theta,c))_{r>0}$ does not tend to $0$ a.s.:
\begin{equation}
\label{NoCVasRPT}
\PP \left( \limsup_{r\to\infty} \chi_{r}(\theta,c) \geq 1 \right) = 1 ~.
\end{equation}
\end{theorem}

Let us remark that the ratio $\chi_{r}/r$ should still tend to $0$ in $L^{1}$ and a.s. in any dimension $d\geq 3$. Indeed, the definition of the RPT and the proofs of Steps 1, 3 and 4 should be extended to any dimension without major changes. Moreover, the approximating directed forest still has no bi-infinite path in dimension $d\geq 3$-- even if it is a tree for $d=3$ and a collection of infinitely many trees for $d\geq 4$. See Theorem 3.1 of \cite{FLT} for the corresponding result.


\subsection{The Euclidean FPP Tree}
\label{sect:EuclFPPtree}

Let $\alpha>0$ be a positive real number. The \textit{Euclidean First-Passage Percolation Tree} $\TT_{\alpha}$ introduced and studied in \cite{HN97,HN01}, is a planar graph whose vertex set is given by the homogeneous PPP $\NN$. Unlike the RPT whose graph structure is local, that of the Euclidean FPP Tree is global and results from a minimizing procedure. Let $X,Y\in\NN$. A \textit{path} from $X$ to $Y$ is a finite sequence $(X_{1},\ldots,X_{n})$ of points of $\NN$ such that $X_{1}=X$ and $X_{n}=Y$. To this path, the weight
$$
\sum_{i=1}^{n-1} |X_{i} - X_{i+1}|^{\alpha}
$$
is associated. Then, a path minimizing this weight is called a \textit{geodesic} from $X$ to $Y$ and is denoted by $\g_{X,Y}$:
\begin{equation}
\label{geodesicFPP}
\g_{X,Y} = \arg\!\min \left\{ \sum_{i=1}^{n-1} |X_{i} - X_{i+1}|^{\alpha} \, , \; n\geq 2 \; \mbox{ and } \; (X_{1},\ldots,X_{n}) \; \mbox{is a path from $X$ to $Y$ } \right\} ~.
\end{equation}
By concavity of $x\mapsto x^{\alpha}$ for $0<\alpha\leq 1$, the geodesic from $X$ to $Y$ coincides with the straight line $[X;Y]$. Since a.s. no three points of $\NN$ are collinear, it is reduced to the trivial path $(X,Y)$. So, from now on, to get nontrivial geodesics, we assume $\alpha>1$.

Existence and uniqueness of the geodesic $\g_{X,Y}$ are a.s. ensured whenever $\alpha>1$. This is Proposition 1.1 of \cite{HN01}. Let $X^{O}$ be the closest Poisson point to the origin $O$. The Euclidean FPP Tree $\TT_{\alpha}$ is defined as the collection $\{\g_{X^{O},X}, X\in\NN\}$. By uniqueness of geodesics, $\TT_{\alpha}$ is a tree rooted at $X^{O}$.

Thanks to Proposition 1.2 of \cite{HN01}, any vertex $X$ of $\TT_{\alpha}$ a.s. has finite degree. Remark also that, unlike the RPT, the outgoing vertex of $X\not= X^{O}$ (i.e. its ancestor $A(X)$) may have a larger Euclidean norm than $X$.

The straight character of the Euclidean FPP Tree is stated in Theorem 2.6 of \cite{HN01} (for $\alpha>1$). It then follows:

\begin{prop}[Theorems 1.8 and 1.9 of \cite{HN01}]
\label{prop:EuclFPPtree-S12}
For any $\alpha>1$, the Euclidean FPP Tree $\TT_{\alpha}$ a.s. satisfies statements $[S1]$ and $[S2]$.
\end{prop}

The definition of the number $\chi_{r}$ of semi-infinite geodesics of the Euclidean FPP Tree $\TT_{\alpha}$ at level $r$ requires to be more precise than in Section \ref{sect:RPT}. Since the vertices of geodesics of $\TT_{\alpha}$ are not sorted w.r.t. their Euclidean norms, geodesics may cross many times any given circle. So, let us consider the graph obtained from $\TT_{\alpha}$ after deleting any geodesic $(X^{O},X_{2},\ldots,X_{n})$ with $n\geq 2$ (except the endpoint $X_{n}$) such that the vertices $X^{O},X_{2},\ldots,X_{n-1}$ belong to the disk $D(O,r)$ but $X_{n}$ is outside. Then, $\chi_{r}$ counts the unbounded connected components of this graph. Now, let $\theta\in[0 ; 2\pi)$ and $c>0$. The random integer $\chi_{r}(\theta,c)$ denotes the number of these unbounded connected components emanating from a vertex $X$ such that the edge $[A(X);X]$ crosses the arc $\arc_{r}(\theta,c)$ of the circle $\CC_{r}$. Here is the sublinearity result satisfied by the Euclidean FPP Tree:

\begin{theorem}
\label{theo:sublinEuclFPPtree}
Let $\theta\in[0 ; 2\pi)$ and $c>0$ be real numbers. Assume $\alpha\geq 2$. Then,
\begin{equation}
\label{sublinFPP}
\lim_{r\to\infty} \EE \frac{\chi_{r}}{r} = 0 \; \mbox{ and } \; \lim_{r\to\infty} \EE \chi_{r}(\theta,c) = 0 ~.
\end{equation}
Furthermore, the sequence $(\chi_{r}(\theta,c))_{r>0}$ does not tend to $0$ a.s.:
\begin{equation}
\label{NoCVasFPP}
\PP \left( \limsup_{r\to\infty} \chi_{r}(\theta,c) \geq 1 \right) = 1 ~.
\end{equation}
\end{theorem}

The hypothesis $\alpha\geq 2$ is added so that the Euclidean FPP Tree $\TT_{\alpha}$ satisfies the noncrossing property given in Lemma 5 of \cite{HN97}: for any vertices $X\not= Y$, the open segments $(X;A(X))$ and $(Y;A(Y))$ do not overlap. This property which also holds for the approximating directed forest $\FF_{\alpha}$, will be crucial to obtain the absence of bi-infinite geodesic in $\FF_{\alpha}$.

\subsection{The Directed LPP Tree}
\label{sect:LPPtree}

The \textit{Directed Last-Passage Percolation Tree} is quite different from the RPT or the Euclidean FPP Trees. Indeed, its vertex set is given by the deterministic grid $\N^{2}$. As a result, its random character comes from random times allocated to vertices of $\N^{2}$. See Martin \cite{Martin06} for a complete survey. A \textit{directed path} from the origin $O$ to a given vertex $z\in\N^{2}$ is a finite sequence of vertices $(z_{0},z_{1},\ldots,z_{n})$ with $z_{0}=O$, $z_{n}=z$ and $z_{i+1}-z_{i}=(1,0)\text{ or } (0,1)$, for $0\leq i\leq n-1$. The time to go from the origin to $z$ along the path $(z_{0},z_{1},\ldots,z_{n})$ is equal to the sum $\omega(z_{0})+\ldots+\omega(z_{n-1})$, where $\{\omega(z), z\in\N^{2}\}$ is a family of i.i.d. positive random variables such that
\begin{equation}
\label{LPP-Hypo1}
\mbox{$\EE\omega(z)^{2+\eps}<\infty$ for some $\eps>0$ and $\textrm{Var}(\omega(z))>0$}
\end{equation}
and
\begin{equation}
\label{LPP-Hypo2}
\mbox{$\PP(\omega(z)\geq r)$ is a continuous function of $r$}.
\end{equation}
A directed path maximizing this time over all directed paths from the origin to $z$ is denoted by $\g_{z}$ and called a \textit{geodesic} from the origin to $z$:
\begin{equation}
\label{geodesicLPP}
\g_{z} = \arg\!\max \left\{ \sum_{i=0}^{n-1} \omega(z_{i}) \, , \; (z_{0},\ldots,z_{n}) \; \mbox{is a directed path from $O$ to $z$ } \right\} ~.
\end{equation}
Hypothesis (\ref{LPP-Hypo2}) ensures the almost sure uniqueness of geodesics. Then, the collection of all these geodesics provides a random tree rooted at the origin and spanning all the quadrant $\N^{2}$. It is called the Directed Last-Passage Percolation Tree and is denoted by $\TT$. See Figure \ref{fig:LLPtree} for an illustration. Given $z\in\N^{2}\setminus\{O\}$, the \textit{ancestor} $A(z)$ of $z$ is the vertex among $z-(1,0)$ and $z-(0,1)$ by which its geodesic passes. The chidren of $z$ are the vertices among $z+(1,0)$ and $z+(0,1)$ whose $z$ is the ancestor.\\

\begin{figure}[!ht]
\begin{center}
\includegraphics[width=10cm,height=10cm]{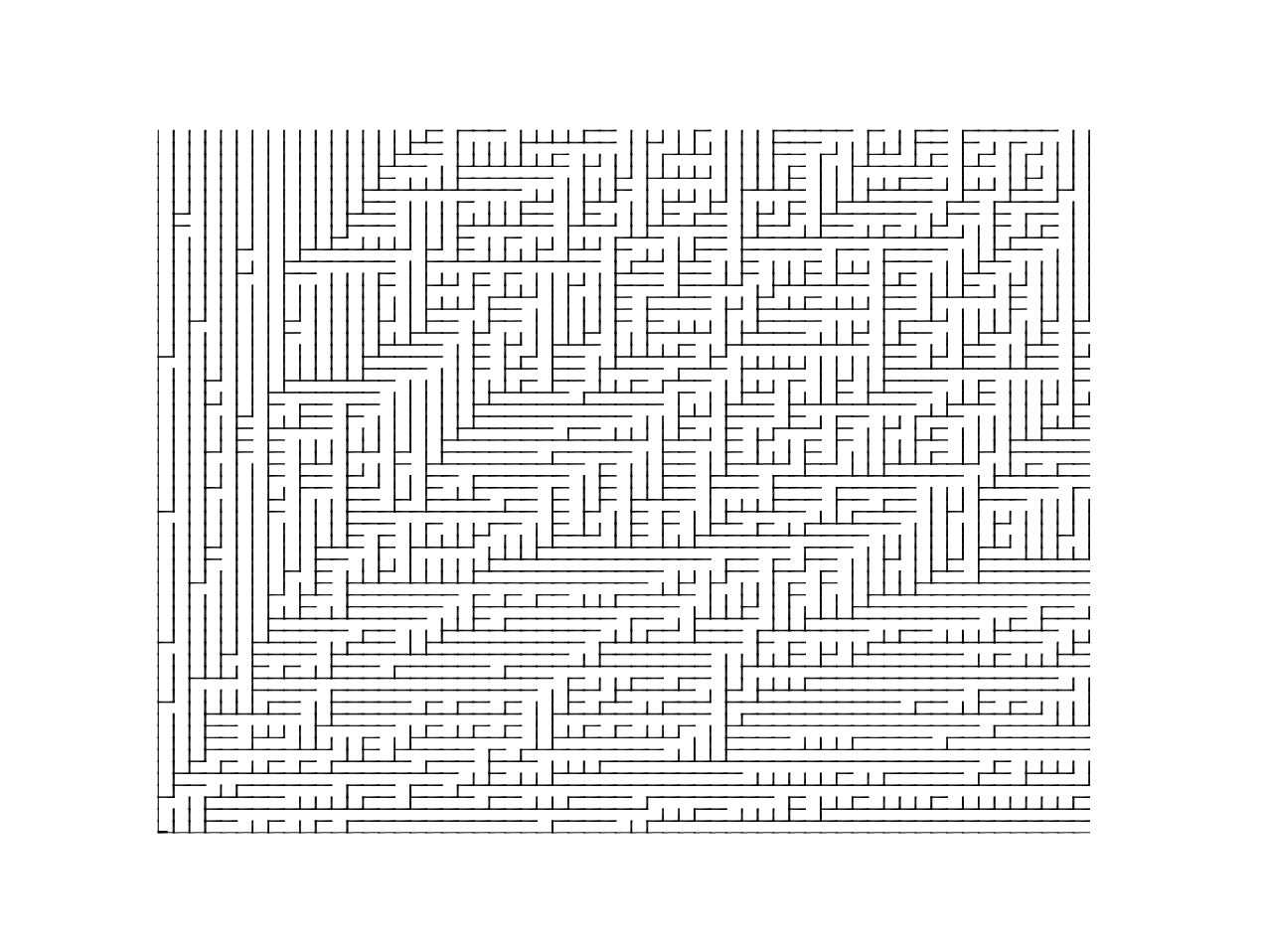}
\end{center}
\caption{\label{fig:LLPtree} Here is a realization of the Directed LPP Tree $\TT$ restricted to the set $[0;60]^{2}$ in the case of the exponential distribution with parameter $1$. This realization presents a remarkable feature; the geodesics to $(60,59)$ and to $(59,60)$ have only one common vertex which is the origin.}
\end{figure}

The study of geodesics of the Directed LPP Tree has started in \cite{FP} with the case of exponential weights and has been recently generalized in \cite{GeRaSe} to a larger class corresponding to (\ref{LPP-Hypo1}). However, a third hypothesis is required so that the Directed LPP Tree $\TT$ satisfies $[S1]$ and $[S2]$. If $G(z)$ denotes the times realized along the geodesic $\g_{z}$, then there exists a nonrandom continuous function $g:\R^{2}_{+}\to\R$ defined by
$$
\mbox{a.s. } g(x) = \lim_{n\to\infty} \frac{G(\lfloor nx\rfloor)}{n}
$$
and called the \textit{shape function}. See for instance Proposition 2.1 of \cite{Martin04}. The shape function $g$ is symmetric, concave and $1$-homogeneous. In the sequel, we also assume that
\begin{equation}
\label{LPP-Hypo3}
\mbox{$g$ is strictly concave.}
\end{equation}

\begin{prop}[Theorem 2.1 of \cite{GeRaSe}]
\label{prop:LPPtree-S12}
With hypotheses (\ref{LPP-Hypo1}), (\ref{LPP-Hypo2}) and (\ref{LPP-Hypo3}), the Directed LPP Tree $\TT$ a.s. satisfies statements $[S1]$ and $[S2]$.
\end{prop}

In the case of exponential weights, the LPP model is deeply linked to the Totally Asymmetric Simple Exclusion Process (TASEP). As a consequence, more precise results exist about semi-infinite geodesics of the Directed LPP Tree: see \cite{C3}.

Let $\theta\in[0 ; \pi/2]$ and $c>0$. The number of intersection points between the arc $\arc_{r}(\theta,c)$ of the circle $\CC_{r}$, and the semi-infinite geodesics of the Directed LPP Tree $\TT$ is denoted by $\chi_{r}(\theta,c)$. Due to its directed character, $\TT$ always contains two trivial semi-infinite geodesics which are the horizontal and vertical axes. This implies that, for any $c>0$, $\chi_{r}(0,c)$ and $\chi_{r}(\pi/2,c)$ are larger than $1$. This is the reason why the extreme values $\theta=0$ and $\theta=\pi/2$ are excluded from the first part of Theorem \ref{theo:sublinLPPtree}.

\begin{theorem}
\label{theo:sublinLPPtree}
Assume (\ref{LPP-Hypo1}), (\ref{LPP-Hypo2}) and (\ref{LPP-Hypo3}) hold. Let $\theta\in(0 ; \pi/2)$ and $c>0$ be real numbers. Then,
\begin{equation}
\label{sublinLPP}
\lim_{r\to\infty} \EE \chi_{r}(\theta,c) = 0 ~.
\end{equation}
Furthermore, for any $\theta\in[0 ; \pi/2]$ and $c>0$, the sequence $(\chi_{r}(\theta,c))_{r>0}$ does not tend to $0$ a.s.:
\begin{equation}
\label{NoCVasLPP}
\PP \left( \limsup_{r\to\infty} \chi_{r}(\theta,c) \geq 1 \right) = 1 ~.
\end{equation}
\end{theorem}

Because of the lack of isotropy of the Directed LPP Tree, we cannot immediatly deduce from (\ref{sublinLPP}) that $\EE \chi_{r}/r$ tends to $0$. In the case of exponential weights, a possible way to overcome this obstacle would be to take advantage of the coupling between the LPP model and the TASEP.

In order to avoid extra definitions, we do not mention explicitly in Theorem \ref{theo:sublinLPPtree} the following extension: (\ref{sublinLPP}) still holds without the strict concacivity of the shape function $g$, i.e. the restriction of $g$ to $\{(t,1-t), 0\leq t\leq 1\}$ may admit flat segments. In this case, thanks to Theorem 2.1 of \cite{GeRaSe}, geodesics are no longer directed according to a given direction but according to a semi-cone (generated by a flag segment). 

\section{Sketch of the proofs}
\label{sect:Sketch}

In this section, we use the generic notation $\TT$ to refer to the RPT $\TT_{\rho}$, to the Euclidean FPP Tree $\TT_{\alpha}$ and to the Directed LPP Tree $\TT$.

\subsection{Sublinearity results and comments}
\label{sect:SketchSublin}

First of all, the isotropic property allows to reduce the study to any given direction $\theta$:
$$
\EE \chi_{r} = r \EE \chi_{r}(\theta,2\pi) ~.
$$
This is the case of the RPT and the Euclidean FPP Tree, but not the Directed LPP Tree. So, our goal is to show that the expectation of $\chi_{r}(\theta,2\pi)$ tends to $0$ as $r$ tends to infinity. The scheme of the proof can be divided into four steps.\\

\textbf{STEP 1:} The first step consists in approximating locally (i.e. around the point $re^{\i \theta}$) and in distribution the tree $\TT$ by a suitable directed forest, say $\FF$, with direction $-e^{\i \theta}$. To do it, we need \textit{local functions}. Let us consider two oriented random graphs $G$ and $G'$ with out-degree $1$ defined on the same probability space and having the same vertex set $V\subset\R^{2}$. As previously, we call ancestor of $v$, the endpoint of the outgoing edge of $v$.

\begin{definition}
\label{def:localfct}
With the previous assumptions, a measurable function $F$ is said \textnormal{local} if there exists a (deterministic) set $D$, called the \textnormal{stabilizing set} of $F$, such that for any $X\in\R^{2}$; $F(X,G)=F(X,G')$ whenever each vertex $v\in V\cap (X+D)$ has the same ancestor in $G$ and $G'$.
\end{definition}

Thenceforth, the approximation result will be expressed as follows. Given a local function $F$, the distribution of $F(re^{\i \theta},\TT)$ converges in total variation towards the distribution of $F(O,\FF)$ as $r$ tends to infinity.

The directed forest $\FF$ approximating the RPT has been introduced by Ferrari \etal \cite{FLT}. This forest is given by the collection of coalescing one-dimensional random walks with uniform jumps in a bounded interval (with radius $\rho$) and starting at the points of a homogeneous PPP in $\R^{2}$. Its graph structure is based on local rules. Conversely, the directed forests used to approximate the Euclidean FPP Tree and the Directed LPP Tree are collections of coalescing semi-infinite geodesics with direction $-e^{\i \theta}$: their graph structures obey to global rules. Consequently, the proofs of Step 1 for the RPT (see Lemma \ref{lem:approxRPT} and comments just after Proposition \ref{prop:RPTApprox}) and for the FPP/LPP Trees will be radically different. In particular, a key argument used in the proof for optimized trees is that a geodesic from $X$ to $Y$ coincides with the one from $Y$ to $X$-- which does not hold in general for greedy trees.\\

\textbf{STEP 2:} The goal of the second step is to prove that the directed forest $\FF$ (with direction $-e^{\i \theta}$) a.s. has no bi-infinite path. Such a proof is now classic. First one states that all the paths eventually coalesce (towards the direction $-e^{\i \theta}$). This can be easily done if the directed forest presents some Markov property; this is the case of the forest approximating the RPT (see Section 4 of \cite{FLT}). Otherwise, an efficient topological argument originally due to Burton and Keane \cite{BK} may apply. See Licea and Newman \cite{LN} for an adaptation of this argument to the FPP/LPP context. This argument is fundamentally based on the fact that paths do not cross in dimension two.

In a second time, we deduce from the coalescence result that $\FF$ does not contain any bi-infinite path. Roughly speaking, when one looks to the past (i.e. towards the direction $e^{\i \theta}$), all the paths are finite. This part essentially uses the translation invariance property of the directed forest $\FF$.

Step 2 also says that the directed forest $\FF$ a.s. has only one topological end.\\

\textbf{STEP 3:} Combining results of the two previous steps, we get that $\chi_{r}(\theta,2\pi)$ tends to $0$ in probability, as $r$ tends to infinity. Let us roughly describe the underlying idea. The event $\chi_{r}(\theta,2\pi)\geq 1$ implies the existence of a very long path of $\TT$ crossing the arc of the circle $\CC_{r}$ centered at $re^{\i\theta}$ and with length $c$. Thanks to Step 1, this means that in the directed forest $\FF$, there exists a path crossing the segment centered at $O$, with length $c$ and orthogonal to $-e^{\i \theta}$, and coming from very far in the past (i.e. towards the direction $e^{\i \theta}$). Now, thanks to Step 2, this should not happen.\\

\textbf{STEP 4:} In this last step, we exhibit a uniform (on $r$) moment condition for $\chi_{r}(\theta,2\pi)$ to strengthen its convergence to $0$ in the $L^{1}$ sense using the Cauchy-Schwarz inequality:
$$
\EE \chi_{r}(\theta,2\pi) = \EE \chi_{r}(\theta,2\pi) \II_{\chi_{r}(\theta,2\pi)\geq 1} \leq M \sqrt{\PP(\chi_{r}(\theta,2\pi)\geq 1)}
$$
for any $r$ large enough.\\

As recalled in Introduction, this method has already been applied to the Radial Spanning Tree (RST) in a series of articles. The RST is locally approximated by the Directed Spanning Forest (Step 1): see Theorem 2.4 of \cite{BB} for the precise result and Section 2.1 for the construction of the DSF. The fact that this forest has no bi-infinite path is the main result of Coupier and Tran \cite{CT} (Step 2). Finally, Steps 3 and 4 are given in Coupier \etal \cite{BCT} and lead to the sublinearity result (Theorem 2).

This method works as well for the RPT $\TT_{\rho}$: the mean number of semi-infinite paths of $\TT_{\rho}$ is asymptotically sublinear. However, in Section \ref{sect:proofRPT}, we perform this method to obtain a better rate of convergence, namely $\EE\chi_{r}$ is $o(r^{3/4+\eps})$. Besides, this performed method which is developped in Section \ref{sect:proofRPT}, should apply to the RST provided a coalescence time estimate exists for its approximating directed forest.

Unlike the FPP/LPP context, the approximation method (STEP 1) for greedy trees (as RPT, RST) does not require the existence of semi-infinite paths in each deterministic direction-- which generally follows from the straight character of the considered tree. Hence, we can try to apply our method to the geometric random tree of Coletti and Valencia \cite{CV} without assuming that it admits an infinite number of semi-infinite paths (which should certainly be true).

Actually, our method says a little bit more. Conditionally to the fact that STEPS 1, 3 and 4 work, $\EE\chi_{r}(\theta,c)$ tends to $0$ if and only if the corresponding directed forest contains no bi-infinite path. Hence, it seems possible to use Example 2.5 of \cite{GeRaSe}, to prove that
$$
\liminf_{r\to\infty} \EE \chi_{r}(\pi/4,1) > 0 ~,
$$
where $\chi_{r}(\pi/4,1)$ concerns the rightmost semi-infinite paths (directed according to a cone with axis $e^{\i\pi/4}$) of some particular Directed LPP Tree satisfying (\ref{LPP-Hypo1}) but not (\ref{LPP-Hypo2}) and (\ref{LPP-Hypo3}).


\subsection{Absence of directional almost sure convergence}

For each of the three random trees studied in this paper, the r.v. $\chi_{r}(\theta,2\pi)$ tends to $0$ in probability but not almost surely. This absence of almost sure convergence is based on the same key result.

\begin{proposition}
\label{prop:directiondeterm}
Almost surely, there is at most one semi-infinite path (or geodesic) of $\TT$ with deterministic direction $\theta$. This statement holds for the RPT with $\theta\in[0 ; 2\pi)$ and $\rho>0$; for the Euclidean FPP Tree with $\theta\in[0 ; 2\pi)$ and $\alpha\geq 2$; for the Directed LPP Tree with $\theta\in[0 ; \pi/2]$ and hypotheses (\ref{LPP-Hypo1}), (\ref{LPP-Hypo2}) and (\ref{LPP-Hypo3}).
\end{proposition}

This result has been proved in Lemma 6 of \cite{HN97} for the Euclidean FPP Tree and in Proposition 5 of \cite{BCT} for the Radial Spanning Tree. In both cases, a clever application of Fubini's theorem allows to get that, for a.e. $\theta$ (w.r.t. the Lebesgue measure), there a.s. is no semi-infinite path with direction $\theta$. Thus, by isotropy, the result can be extended to any $\theta$. Actually, the same proof works for the RPT. For this reason, we do not give any detail.

Proposition \ref{prop:directiondeterm} is given in Theorem 2.1 (iii) of \cite{GeRaSe} for Directed LPP Tree. Their proof, completely different from the previous argument, uses cocycles to overcome the lack of isotropy.\\

It remains to prove
$$
\PP \left( \limsup_{r\to\infty} \chi_{r}(\theta,c) \geq 1 \right) = 1 ~.
$$
from Proposition \ref{prop:directiondeterm}. This has been already written into details in \cite{BCT} in the case of the RST (see Corollary 6). Without major changes, the same arguments work for the RPT, the Euclidean FPP Tree and the Directed LPP Tree. Hence, we will just describe the spirit of the proof. By contradiction, let us assume that with positive probability, from a (random) radius $r_{0}$, there is no semi-infinite path of the tree $\TT$ crossing the arc $\arc_{r}(\theta,c)$ of the circle $\CC_{r}$ (with a deterministic direction $\theta$). Hence, with positive probability, there is no semi-infinite path crossing the semi-line $L(\theta,r_{0})=\{re^{\i\theta}, r\geq r_{0}\}$. Now, from both unbounded subtrees of $\TT$ located on each side of the semi-line $L(\theta,r_{0})$, it is possible to extract two semi-infinite paths, say $\g$ and $\g'$, which are as close as possible to $L(\theta,r_{0})$. This construction ensures that the region of the plane delimited by $\g$ and $\g'$-- in which the semi-line $L(\theta,r_{0})$ is --only contains finite paths of $\TT$. Since the tree $\TT$ satisfies statements $[S1]$ and $[S2]$, we can then deduce that $\g$ and $\g'$ have the same asymptotic direction $\theta$. However, such a situation never happens by Proposition \ref{prop:directiondeterm}.

\section{Convergence in $L^{1}$ for the Euclidean FPP Tree}
\label{sect:proofEuclFPPtree}

To get Theorem \ref{theo:sublinEuclFPPtree}, it suffices by isotropy to prove that the expectation of $\chi_{r}(0,2\pi)$ tends to $0$ as $r\to\infty$. The proof works as well when $2\pi$ is replaced with any constant $c>0$. In the sequel, we assume $\alpha\geq 2$.\\

\textbf{STEP 1:} Proposition \ref{prop:directiondeterm} combined with statement $[S2]$ of Proposition \ref{prop:EuclFPPtree-S12} says the Euclidean FPP Tree $\TT_{\alpha}$, rooted at $X^{O}$, a.s. contains exactly one semi-infinite geodesic with direction $\pi$. Actually, this argument applies to each Euclidean FPP Tree rooted at any $X\in\NN$ (for the same parameter $\alpha$). Hence, we denote by $\gamma^{\infty}_{X}$ the semi-infinite geodesic with direction $\pi$ of the Euclidean FPP Tree rooted at $X$. Let $\mathcal{F}_{\alpha}$ be the collection of the $\gamma^{\infty}_{X}$'s, for all $X\in\NN$. By uniqueness of geodesics, $\mathcal{F}_{\alpha}$ is a directed forest with direction $\pi$ which is built on the PPP $\NN$. In the geodesic $\gamma^{\infty}_{X}$, the neighbor of $X$ is called its \textit{ancestor} (in $\mathcal{F}_{\alpha}$), and is denoted by $\bar{A}(X)$. Be careful, the ancestor $\bar{A}(X)$ of $X$ is the vertex whose $X$ is the ancestor in the Euclidean FPP Tree rooted at $X$.

Our goal is to approximate the Euclidean FPP Tree $\TT_{\alpha}$ around $(r,0)$ by the directed forest $\mathcal{F}_{\alpha}$:

\begin{proposition}
\label{prop:EuclFPPtreeStep1}
Let $F$ be a local function whose stabilizing set is $D(O,L)$ (see Definition \ref{def:localfct}) with $L>0$. Then,
$$
\lim_{r\to\infty} d_{TV} \Big( F((r,0),\TT_{\alpha}) , F(O,\FF_{\alpha}) \Big) = 0 ~,
$$
where $d_{TV}$ denotes the total variation distance.
\end{proposition}

It is important to notice that the parameter $L$ occurring in the stabilizing set of $F$ does not depend on $r$.

Unlike the Directed Poisson Forest $\mathcal{F}_{\rho}$ used in Section \ref{sect:proofRPT} to approximate the RPT $\TT_{\rho}$, the graph structure of $\mathcal{F}_{\alpha}$ is clearly non local. Hence, the proof of Proposition \ref{prop:EuclFPPtreeStep1} will be radically different to the one about the RPT (see the paragraph below Proposition \ref{prop:RPTApprox}).

\begin{proof}
By the translation invariance property of the directed forest $\FF_{\alpha}$, we can write:
\begin{eqnarray*}
d_{TV} \left( F((r,0),\TT_{\alpha}) , F(O,\FF_{\alpha}) \right) & = & d_{TV} \left( F((r,0),\TT_{\alpha}) , F((r,0),\FF_{\alpha}) \right) \\
& \leq & \PP \left( F((r,0),\TT_{\alpha}) \not= F((r,0),\FF_{\alpha}) \right) \\
& \leq & \PP \left( \exists X \in \NN\cap D((r,0),L) , \, A(X) \not= \bar{A}(X) \right) ~.
\end{eqnarray*}
We now consider $\TT_{\alpha}$ and $\FF_{\alpha}$ built on the same vertex set $\NN$. Since the ancestors of $X$ differ in the Euclidean FPP Tree $\TT_{\alpha}$ (rooted at $X^{O}$) and in $\FF_{\alpha}$, the geodesics $\gamma_{X^{thanO},X}$ and $\gamma^{\infty}_{X}$ have only the vertex $X$ in common. Moreover, for any $\eps>0$, the root $X^{O}$ belongs to the disk $D(O,r^{\eps})$ with a probability tending to $1$. So it suffices to state that
\begin{equation}
\label{ApproxLocalFPP-1}
\lim_{r\to\infty} \PP \left( \exists X \in \NN\cap D((r,0),L) , \, \gamma_{X^{O},X}\cap\gamma^{\infty}_{X} = \{X\} \; \mbox{ and } \; |X^{O}|\leq r^{\eps} \right) = 0 ~.
\end{equation}

A key remark is that the geodesic from $X^{O}$ to $X$ (in the Euclidean FPP Tree rooted at $X^{O}$) coincides with the geodesic from $X$ to $X^{O}$ (in the Euclidean FPP Tree rooted at $X$). It is worth pointing out here this property does not hold in the RPT context. Hence, by translation invariance, the probability in (\ref{ApproxLocalFPP-1}) is bounded by
\begin{equation}
\label{ApproxLocalFPP-2}
\PP \left( \exists X \in \NN\cap D(O,L) , \, \exists X' \in \NN\cap D((-r,0),r^{\eps}) , \, \gamma_{X,X'}\cap\gamma^{\infty}_{X} = \{X\} \right) ~.
\end{equation}

The idea to prove that (\ref{ApproxLocalFPP-2}) tends to $0$ can be expressed as follows. By Lemmas \ref{lem:Cone-gamma(X)} and \ref{lem:Cone-T<out>} respectively, both geodesics $\gamma^{\infty}_{X}$ and $\gamma_{X,X'}$ are included in a cone with direction $(-1,0)$. However, having two long geodesics with a common deterministic direction should not happen according to Proposition \ref{prop:directiondeterm}.

Let $C(Y,\eta)=\{Y'\in\R^{2}, \theta(Y,Y')\leq\eta\}$ where $\theta(Y,Y')$ is the absolute value of the angle (in $[0;\pi]$) between $Y$ and $Y'$ and let $\gamma^{\infty}_{X}(M)$ be the geodesic $\gamma^{\infty}_{X}$ restricted to $D(O,M)^{c}$. Lemma \ref{lem:Cone-gamma(X)} says that, with high probability, $\gamma^{\infty}_{X}(M)$ is included in the cone $C((-1,0),\eta)$ for $M$ large enough.

\begin{lemma}
\label{lem:Cone-gamma(X)}
For all $\eta>0$,
$$
\lim_{M\to\infty} \PP \left( \forall X \in \NN\cap D(O,L) , \, \gamma^{\infty}_{X}(M) \; \mbox{ is included in } \; C((-1,0),\eta) \right) \, = \, 1 ~.
$$
\end{lemma}

Now, let us proceed by contradiction and assume that the probability (\ref{ApproxLocalFPP-2}) does not tend to $0$. Thanks to Lemma \ref{lem:Cone-gamma(X)}, we can assert that, for $\eta>0$, there exist constants $c>0$ and $M$ large enough so that
\begin{equation}
\label{ApproxLocalFPP-3}
\PP \left(\begin{array}{c}
\exists X \in \NN\cap D(O,L) , \, \exists X' \in \NN\cap D((-r,0),r^{\eps}) \\
\mbox{such that } \, \gamma_{X,X'}\cap\gamma^{\infty}_{X} = \{X\} \\
\mbox{ and $\gamma^{\infty}_{X}(M)$ is included in } \, C((-1,0),\eta)
\end{array}\right) \, \geq \, c
\end{equation}
where $r$ can be chosen as large as we want.

Let $\TT^{\mbox{\tiny{out}}}_{\alpha,X}(Y)$ be the subtree rooted at $Y$ of the the Euclidean FPP Tree rooted at $X$. In other words, $\TT^{\mbox{\tiny{out}}}_{\alpha,X}(Y)$ is the collection of geodesics starting at $X$ and passing by $Y$, whose common part-- from $X$ to $Y$ --has been deleted. Lemma \ref{lem:Cone-T<out>} asserts that the subtrees $\TT^{\mbox{\tiny{out}}}_{\alpha,X}(Y)$ of any Euclidean FPP Tree rooted at a vertex $X$ in the disk $D(O,L)$ are becoming thinner as their root $Y$ is far away from the origin.

\begin{lemma}
\label{lem:Cone-T<out>}
For any $\eta>0$ small enough and for any $M$ large enough,
$$
\PP \left(\begin{array}{c}
\forall X \in \NN\cap D(O,L) , \, \forall Y \in \NN\cap D(O,M)^{c} , \\
\TT^{\mbox{\tiny{out}}}_{\alpha,X}(Y) \; \mbox{ is included in } \; C(Y,\eta)
\end{array}\right) \, \geq \, 1 - \frac{c}{2}
$$
where $c$ is the constant given in (\ref{ApproxLocalFPP-3}).
\end{lemma}

From now on, we also set $\eps<1$ so that $D((-r,0),r^{\eps})$ is included in the cone $C((-1,0),\frac{\eta}{2})$ for $r$ large enough. Lemma \ref{lem:Cone-T<out>} implies that with high probability, the geodesic $\gamma_{X,X'}$ is also included in the cone $C((-1,0),\eta)$. Otherwise, this geodesic would pass by a Poisson point $Z$ outside the cone $C((-1,0),\eta)$ and by $X'$ which belongs to $D((-r,0),r^{\eps})$. For $r$ large enough, this would imply that the subtree $\TT^{\mbox{\tiny{out}}}_{\alpha,X}(Z)$ which contains $X'$ is not included in $C((-1,0),\frac{\eta}{2})$. As a consequence, for $M$ large enough,
\begin{equation}
\label{ApproxLocalFPP-4}
\PP \left(\begin{array}{c}
\exists X \in \NN\cap D(O,L) , \, \exists X' \in \NN\cap D((-r,0),r^{\eps}) \\
\mbox{such that } \; \gamma_{X,X'}\cap\gamma^{\infty}_{X} = \{X\} \, , \; \gamma_{X,X'}(M) \\
\mbox{ and $\gamma^{\infty}_{X}(M)$ are included in } \; C((-1,0),\eta)
\end{array}\right) \, \geq \, \frac{c}{2}
\end{equation}
where $r$ can be chosen as large as we want. Above, $\gamma_{X,X'}(M)$ denotes the geodesics $\gamma_{X,X'}$ restricted to $D(O,M)^{c}$. Now, the interpreted event in (\ref{ApproxLocalFPP-4}) and described in Figure \ref{fig:FPPdirdeterm} implies that one can find an Euclidean FPP Tree, rooted at a given $X$ in $D(O,L)$, from which it is possible to extract two geodesics included in $C((-1,0),\eta)$ and as long as we want. By Proposition \ref{prop:directiondeterm}, the probability of such an event must tend to $0$ with $r$. This contradicts (\ref{ApproxLocalFPP-4}).
\end{proof}

\begin{figure}[!ht]
\begin{center}
\psfrag{b1}{\small{$D(O,L)$}}
\psfrag{b2}{\small{$D((-r,0),r^{\eps})$}}
\psfrag{b3}{\small{$D(O,M)$}}
\psfrag{g1}{\small{$\gamma_{X,X'}$}}
\psfrag{g2}{\small{$\gamma^{\infty}_{X}$}}
\psfrag{x1}{\small{$X$}}
\psfrag{x2}{\small{$X'$}}
\includegraphics[width=11cm,height=5cm]{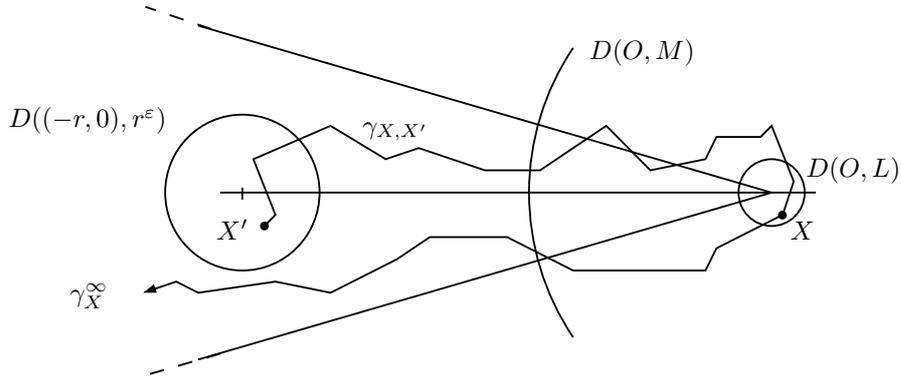}
\end{center}
\caption{\label{fig:FPPdirdeterm} This picture represents the event interpreted in (\ref{ApproxLocalFPP-4}). Poisson points $X$ and $X'$ respectively belong to $D(O,M)$ and $D((-r,0),r^{\eps})$. For $\eps\in(0;1)$ and $r$ large enough, this later disk is included in the cone $C((-1,0),\eta)$. Both geodesics $\gamma_{X,X'}$ and $\gamma^{\infty}_{X}$ restricted to $D(O,M)^{c}$ are inside $C((-1,0),\eta)$.}
\end{figure}

Step 1 ends with the proofs of Lemmas \ref{lem:Cone-gamma(X)} and \ref{lem:Cone-T<out>}.

\begin{proof}{\textit{(of Lemma \ref{lem:Cone-gamma(X)})}}
Let $\eta,\eps$ some positive real numbers and $n$ an integer such that the probability $\PP(\NN(D(O,L))\leq n)$ is larger than $1-\eps$. On the event $\NN(D(O,L))=k$, with $k\leq n$, let us denote by $X_{1},\ldots,X_{k}$ the $k$ vertices of the PPP $\NN$ inside the disk $D(O,L)$. On this event, for $1\leq i\leq k$, the semi-infinite geodesic $\gamma^{\infty}_{X_{i}}$ with direction $\pi$ and starting at $X_{i}$ is included in the cone $C((-1,0),\eta)$ far away from the origin:
$$
\lim_{M\to\infty} \PP \Big( \mbox{$\gamma^{\infty}_{X_{i}}(M)$ is included in $C((-1,0),\eta$)} \Big) = 1 ~.
$$
Then it is possible to choose $M$ large enough so that, for any $i$, the conditional probability
$$
\PP ( \mbox{$\gamma^{\infty}_{X_{i}}(M)$ is included in $C((-1,0),\eta$)} \;|\; \NN (D((r,0),L)) = k )
$$
is larger than $1-\frac{\eps}{n}$. Henceforth,
$$
\PP \left( \left. \begin{array}{c}
\gamma^{\infty}_{X_{1}}(M) , \ldots , \gamma^{\infty}_{X_{k}}(M) \\
\mbox{are included in $C((-1,0),\eta)$}
\end{array} \; \right| \; \NN (D((r,0),L)) = k \right) \; \geq \; 1 - \eps ~,
$$
for $M$ large enough, and then
$$
\PP \left( \begin{array}{c}
\forall X \in \NN\cap D(O,L) , \, \gamma^{\infty}_{X}(M) \\
\mbox{is included in $C((-1,0),\eta)$}
\end{array}\right) \; \geq (1 - \eps) \sum_{k=0}^{n} \PP (\NN (D((r,0),L)) = k) \geq (1 - \eps)^{2} ~.
$$
\end{proof}

\begin{proof}{\textit{(of Lemma \ref{lem:Cone-T<out>})}}
The proof is very close to the one of Lemma \ref{lem:Cone-gamma(X)}. We first restrict our attention to a finite number of vertices inside the disk $D(O,L)$. The Euclidean FPP Tree rooted at one of them, say $X$, is a.s. straight. This implies that:
$$
\lim_{M\to\infty} \PP \Big( \forall Y \in \NN\cap D(O,M)^{c} , \; \TT^{\mbox{\tiny{out}}}_{\alpha,X}(Y) \; \mbox{ is included in } \; C(Y,\eta) \Big) \; = \; 1
$$
from which it is not difficult to conclude.
\end{proof}

\textbf{STEP 2:} The fact that the directed forest $\FF_{\alpha}$ with direction $\pi$ a.s. has no bi-infinite geodesic has been proved in Theorem 1.12 of \cite{HN01} for $\alpha\geq 2$. This means that with probability $1$, the progeny of any $X\in\NN$, i.e. the set $\{Y\in\NN: X\in\gamma^{\infty}_{Y}\}$, is finite. The hypothesis $\alpha\geq 2$ is crucial here since it assures the noncrossing path property. Without this property, we are not able to prove the absence of bi-infinite geodesic in $\FF_{\alpha}$.\\

\textbf{STEP 3:} The goal of this step is to prove that the probability for $\chi_{r}(0,2\pi)$ to be larger than $1$ tends to $0$. It is based on Steps 2 and 3.\\

Let $I_{r}$ be the vertical segment centered at $(r,0)$ and with length $2\pi$. The Hausdorff distance between $\arc_{r}(0,2\pi)$ and $I_{r}$ tends to $0$ with $r$. So, with probability tending to $1$, any path of the Euclidean FPP Tree $\TT_{\alpha}$ crossing $\arc_{r}(0,2\pi)$ also crosses $I_{r}$. Hence, for any $\eps,R>0$ and $r$ large enough,
\begin{equation}
\label{EuclFPPtreestep3-1}
\PP( \chi_{r}(0,2\pi) \geq 1 ) \; \leq \; \PP \left(\begin{array}{c}
\mbox{$\exists$ a geodesic of $\TT_{\alpha}$ crossing $I_{r}$ and} \\
\mbox{afterwards leaving $D((r,0),R)$}
\end{array}\right) + \eps ~.
\end{equation}
The interpreted event mentioned in the r.h.s. of (\ref{EuclFPPtreestep3-1}) means one can extract from $\TT_{\alpha}$ a geodesic $(X_{1},\ldots,X_{\kappa})$ whose vertices $X_{1},\ldots,X_{\kappa-1}$ belong to the disk $D((r,0),R)$, but not $X_{\kappa}$, and the directed edges $(X_{1},X_{2})$ and $(X_{\kappa-1},X_{\kappa})$ respectively cross $I_{r}$ and the circle $\CC((r,0),R)$.

The approximation of $\TT_{\alpha}$ by $\FF_{\alpha}$ (i.e. Proposition \ref{prop:EuclFPPtreeStep1}) requires the use of local functions. It is the reason why we need to control the location of $X_{\kappa}$.

\begin{lemma}
\label{lem:unifEuclFPP}
Let us consider the event $A_{r,R}$ corresponding to ``Each edge $(X,A(X))$ of $\TT_{\alpha}$ s.t. $A(X)$ belongs to the disk $D((r,0),R)$ satisfies $X\in D((r,0),2R)$''. Then, as $R\to\infty$, its probability tends to $1$ uniformly on $r$.
\end{lemma}

The proof of Lemma \ref{lem:unifEuclFPP} is based on the same arguments used and detailled in Step 4 below. For this reason, it is omitted.

Lemma \ref{lem:unifEuclFPP} leads to:
\begin{equation}
\label{EuclFPPtreestep3-2}
\PP( \chi_{r}(0,2\pi) \geq 1 ) \; \leq \; \PP \left(\begin{array}{c}
\mbox{$\exists$ a geodesic $(X_{1},\ldots,X_{\kappa})$ in $\TT_{\alpha}$ such that} \\
\mbox{$X_{1},\ldots,X_{\kappa-1}\in D((r,0),R)$, $(X_{1},X_{2})\cap I_{r}\not=\emptyset$} \\
\mbox{and $X_{\kappa}\in D((r,0),2R)\setminus D((r,0),R)$.} \\
\end{array}\right) + 2 \eps ~,
\end{equation}
for $R$ and $r$ large enough. Let us remark that the uniform limit given by Lemma \ref{lem:unifEuclFPP} implies that up to now the parameters $r$ and $R$ are free from each other.

The interpreted event mentioned in the r.h.s. of (\ref{EuclFPPtreestep3-2}) can be written using a local function whose stabilizing set is a disk with radius $2R$. Then, Proposition \ref{prop:EuclFPPtreeStep1} implies that
\begin{equation}
\label{EuclFPPtreestep3-3}
\PP( \chi_{r}(0,2\pi) \geq 1 ) \; \leq \; \PP \left(\begin{array}{c}
\mbox{$\exists$ a geodesic $(X_{1},\ldots,X_{\kappa})$ of $\FF_{\alpha}$ such that} \\
\mbox{$X_{1},\ldots,X_{\kappa-1}\in D(O,R)$, $(X_{1},X_{2})\cap I_{0}\not=\emptyset$} \\
\mbox{and $X_{\kappa}\notin D(O,R)$.} \\
\end{array}\right) + 3 \eps ~,
\end{equation}
for $r\geq r^{\ast}(R)$ (where $I_{0}=\{0\}\!\times\![-\pi ; \pi]$). Now, the interpreted event in the r.h.s. just above provides the existence of a geodesic in $\FF_{\alpha}$ crossing the vertical segment $I_{0}$ and which is as long as we want in the backward sense, i.e. toward the progeny. Such an event has a probability smaller than $\eps$ thanks to Step 2 for $R$ large enough: $\PP(\chi_{r}(0,2\pi)\geq 1)\leq 4\eps$.\\

\textbf{STEP 4:} In order to strengthen the convergence in probability given by Step 3 into a convergence in $L^{1}$, it is sufficient to prove that
\begin{equation}
\label{EuclFPP-UnifMomCond}
\limsup_{r\to\infty} \EE \chi_{r}(0,2\pi)^{2} \, < \, \infty ~,
\end{equation}
and then to apply the Cauchy-Schwarz inequality. So, let us denote by $\psi_{r}$ the number of edges of the Euclidean FPP Tree $\TT_{\alpha}$ crossing the arc $\arc_{r}(0,2\pi)$ of $\CC_{r}$. Since $\chi_{r}(0,2\pi)\leq\psi_{r}$, we aim to prove that $\PP(\psi_{r}>n)$ decreases exponentially fast (and uniformly on $r$).

Let $R>0$ be a (large) real number. If all the edges counting by $\psi_{r}$ have their endpoints inside the disk $D((r,0),R)$, then $\psi_{r}>n$ forces the PPP $\NN$ to have more than $n$ vertices in $D((r,0),R)$ (otherwise this would contradict the uniqueness of geodesics). This event occurs with small probability (by Lemma \ref{lem:Talagrand}):
\begin{equation}
\label{step4FPP-1}
\PP ( \NN (D((r,0),R)) > n ) \leq e^{-n \ln \left( \frac{n}{e\pi R^{2}} \right)} ~.
\end{equation}
Assume now that (at least) one edge crossing the arc $\arc_{r}(0,2\pi)$ admits one endpoint outside the disk $D(O,R)$. Such a long edge creates a large disk avoiding the PPP $\NN$. Indeed, for any given vertices $X,Y$, if the geodesic $\gamma_{X,Y}$ is reduced to the edge $\{X,Y\}$ then the disk with diameter $[X;Y]$ does not meet the PPP $\NN$. This crucial remark appears in the proof of Lemma 5 in \cite{HN97} and requires $\alpha\geq 2$. To conclude it remains to exhibit an empty deterministic region. To do it, we can consider $\delta_{R}=\lfloor\pi R\rfloor+1$ disks $D_{1},\ldots,D_{\delta_{R}}$ with radius $R/3$ and centered at $\delta_{R}$ points of the circle $\CC((r,0),R/2)$. These $\delta_{R}$ centers can be chosen so that two consecutive ones are at distance smaller than $1$. Therefore, the existence of one edge crossing the arc $\arc_{r}(0,2\pi)$ and having one endpoint outside $D(O,R)$ forces (at least) one of the $D_{i}$'s to avoid the PPP $\NN$. This occurs with a probability smaller than
\begin{equation}
\label{step4FPP-2}
\left( \lfloor \pi R\rfloor + 1 \right) e^{-\pi(R/3)^{2}} ~.
\end{equation}
It remains to take $R=n^{1/4}$ (for instance) so that the upper bounds (\ref{step4FPP-1}) and (\ref{step4FPP-2}) tends to $0$ as $n\to\infty$ uniformly on $r$.

\section{Directional convergence in $L^{1}$ for the LPP Tree}
\label{sect:proofLPP}

Given an angle $\theta\in(0 ; \pi/2)$, recall that $\arc_{r}(\theta,1)$ is the arc of $\CC_{r}$ centered at $re^{\i\theta}$ and with length $1$ (replacing $1$ by any positive constant does not change the proof). Our goal is to show that the mean number $\EE \chi_{r}(\theta,1)$ of semi-infinite geodesics of the Directed LPP Tree $\TT$ crossing the arc $\arc_{r}(\theta,1)$ tends to $0$ as $r$ tends to infinity.\\

\textbf{STEP 1:} Let us introduce a directed forest with direction $\theta+\pi$ defined on the whole set $\Z^{2}$. To do it, we first extend from $\N^{2}$ to $\Z^{2}$ the collection of i.i.d. random weights $\omega(z)$. Replacing the orientation NE with SW, we can define as in Section \ref{sect:LPPtree} and for each $z\in\Z^{2}$, the SW-Directed LPP Tree on the quadrant $z-\N^{2}$. Such a tree a.s. admits exactly one semi-infinite geodesic with direction $\theta+\pi$, say $\gamma(z)$ (i.e. Statement $[S2]$ and Proposition \ref{prop:directiondeterm} hold). Then, we denote by $\FF$ the collection of these semi-infinite geodesics $\gamma(z)$ starting at each $z\in\Z^{2}$. By uniqueness of geodesics, $\FF$ is a forest. Moreover, each vertex $z$ has at most $3$ neighbors; one ancestor (among $z-(1,0)$ and $z-(0,1)$) and $0$, $1$ or $2$ children (among $z+(1,0)$ and $z+(0,1)$).

The directed forest $\FF$ with direction $\theta+\pi$ allows to locally approximate the Directed LPP Tree $\TT$ around $re^{\i\theta}$. The proof of Proposition \ref{prop:LPPtreeStep1} is based on the same ideas as the one of Proposition \ref{prop:EuclFPPtreeStep1} in the Euclidean FPP context. Especially, the SW-geodesic from $z$ to $z'$ coincides with the NE-geodesic from $z'$ to $z$. Indeed, if $(z_{0}=z',z_{1},\ldots,z_{n}=z)$ denotes the geodesic from $z'$ to $z$ then the sum $\omega(z_{0})+\omega(z_{1})+\ldots +\omega(z_{n-1})$ is maximal among NE-paths from $z'$ to $z$ if and only if $\omega(z_{1})+\ldots +\omega(z_{n})$ is maximal among SW-paths from $z$ to $z'$. So we do not give the proof.

\begin{proposition}
\label{prop:LPPtreeStep1}
Let $F$ be a local function. Then,
$$
\lim_{r\to\infty} d_{TV} \Big( F(re^{\i\theta},\TT) , F(O,\FF) \Big) = 0 ~.
$$
\end{proposition}

\textbf{STEP 2:} Thanks to Theorem 2.1 (iii) of \cite{GeRaSe}, we already know that, with probability $1$, $\FF$ has no bi-infinite geodesic.\\

\textbf{STEP 3:} The proof that $\chi_{r}(\theta,1)$ tends to $0$ in probability, is exactly the same as in Section \ref{sect:proofEuclFPPtree}. Actually, some technical simplications arise because the edges of the Directed LPP Tree $\TT$ all are of length $1$.\\

\textbf{STEP 4:} In order to strengthen the convergence in probability given by Step 3 into a convergence in $L^{1}$, we need to control the number of edges of the Directed LPP Tree $\TT$ crossing the arc $\arc_{r}(\theta,1)$. Since the vertex set of $\TT$ is $\N^{2}$, this is automatically fulfilled.

\section{Convergence in $L^{1}$ for the RPT}
\label{sect:proofRPT}

Let $0<\alpha<1/4$. By isotropy, it suffices to prove that $\EE\chi_{r}(0,2r^{\alpha})$ tends to $0$ as $r\to\infty$. Recall that $\chi_{r}(0,2r^{\alpha})$ counts the intersection points between the semi-infinite paths of the RPT $\TT_{\rho}$ and the arc $\arc_{r}(\theta,2r^{\alpha})$ of the circle $\CC_{r}$.\\

Let us consider the rectangle
$$
\Rect(r,\beta,\eps) = [r,r+r^{\beta}]\!\times\![-r^{\beta/2+\eps},r^{\beta/2+\eps}] ~,
$$
where $\beta,\eps$ are positive real numbers. Let us also introduce the r.v. $\chi_{r}(\alpha,\beta,\eps)$ which counts the intersection points between the vertical segment $I_{r}=\{r\}\times[-r^{\alpha},r^{\alpha}]$ and paths $\gamma=(X_{1},\ldots,X_{n})$ of $\TT_{\rho}$ such that: $A(X_{i})=X_{i+1}$ for all $1\leq i\leq n-1$; $\gamma$ starts from the outside of $\Rect(r,\beta,\eps)$, i.e. $X_{1}\notin\Rect(r,\beta,\eps)$; and $\gamma$ crosses $I_{r}$ from right to left, i.e. $[X_{n-1};X_{n}]\cap I_{r}\not=\emptyset$ and $X_{n}(1)<r<X_{n-1}(1)$. Let us point out here that nothing forbids edges of $\TT_{\rho}$ to cross $I_{r}$ from left to right.

We first claim that:

\begin{lemma}
\label{lem:RPT-Claim1}
With the previous notations,
$$
\limsup_{r\to\infty} \EE \chi_{r}(0,2r^{\alpha}) \leq \limsup_{r\to\infty} \EE \chi_{r}(\alpha,\beta,\eps) ~.
$$
\end{lemma}

\begin{proof}
Let $\tilde{\chi}_{r}(\alpha,\beta,\eps)$ be the number of edges crossing the segment $I_{r}$ from right to left and belonging to semi-infinite paths of $\TT_{\rho}$. Since $\tilde{\chi}_{r}(\alpha,\beta,\eps)\leq \chi_{r}(\alpha,\beta,\eps)$ a.s. it is then sufficient to show that
$$
\EE|\chi_{r}(0,2r^{\alpha})-\tilde{\chi}_{r}(\alpha,\beta,\eps)| \to 0 ~,
$$
as $r\to\infty$. The difference $|\chi_{r}(0,2r^{\alpha})-\tilde{\chi}_{r}(\alpha,\beta,\eps)|$ is bounded from above by the number of edges of $\TT_{\rho}$ crossing one of the segments $[A^{+};B^{+}]$ or $[A^{-};B^{-}]$ where $A^{+}$ and $A^{-}$ (resp. $B^{+}$ and $B^{-}$) are the two endpoints of the arc $\arc_{r}(\theta,2r^{\alpha})$ (resp. the segment $I_{r}$)-- with $A^{+}(2)>0$ and $B^{+}(2)=r^{\alpha}$. As $r\to\infty$, $|A^{+}-B^{+}|$ and $|A^{-}-B^{-}|$ tends to $0$. So, the mean number of edges crossing one of the segments $[A^{+};B^{+}]$ or $[A^{-};B^{-}]$ also tends to $0$ as $r\to\infty$.
\end{proof}

The second step consists in approximating the Radial Poisson Tree $\TT_{\rho}$ in the direction $\theta=0$, i.e. in the vicinity of $(r,0)$, by the directed forest $\FF_{\rho}$ with direction $-(1,0)$ introduced by Ferrari \etal in \cite{FLT}. First, let us recall the graph structure of the forest $\FF_{\rho}$ built on the PPP $\NN$. Each vertex $X\in\NN$ is linked to the element of 
$$
\NN \cap \Big( X + (-\infty ; 0)\!\times\! [-\rho ; \rho] \Big)
$$
having the largest abscissa. It is a.s. unique, called the \textit{ancestor} of $X$ and denoted by $\A(X)$. By construction, the sequence of ancestors $(X_{0},X_{1},X_{2}\ldots)$ starting from any vertex $X$ in which $X_{0}=X$ and $\A(X_{n})=X_{n+1}$ for any $n$, is a semi-infinite path denoted by $\gamma^{\infty}_{X}$.

Thus, let us consider the r.v. $\eta_{r}(\alpha,\beta,\eps)$ which is the analogue of $\chi_{r}(\alpha,\beta,\eps)$ but for the directed forest $\FF_{\rho}$. Precisely, $\eta_{r}(\alpha,\beta,\eps)$ counts the  intersection points between $I_{r}$ and paths of $\FF_{\rho}$ starting from the outside of $\Rect(r,\beta,\eps)$. Remark that such paths necessarily cross $I_{r}$ from right to left.

Our second claim is:

\begin{lemma}
\label{lem:RPT-Claim2}
Assume $\alpha\leq\beta/2$ and $\beta+\eps<1/2$. With the previous notations,
$$
\limsup_{r\to\infty} \EE \chi_{r}(\alpha,\beta,\eps) \leq \limsup_{r\to\infty} \EE \eta_{r}(\alpha,\beta,\eps) ~.
$$
\end{lemma}

The proof of Lemma \ref{lem:RPT-Claim2} requires the following approximation result which is stronger than Proposition \ref{prop:EuclFPPtreeStep1} of Section \ref{sect:proofEuclFPPtree} or Proposition \ref{prop:LPPtreeStep1} of Section \ref{sect:proofLPP}. On the one hand, this approximation of $\TT_{\rho}$ by $\FF_{\rho}$ is no longer local since the size of the rectangle $\Rect(r,\beta,\eps)$ goes to infinity with $r$. On the other hand, this approximation result concerns the mean number of errors between $\TT_{\rho}$ and $\FF_{\rho}$ and not only the probability that at least one error occurs.

\begin{proposition}
\label{prop:RPTApprox}
Assume $\alpha\leq\beta/2$ and $\beta+\eps<1/2$. Then,
\begin{equation}
\label{RPTApproxMeanError}
\lim_{r\to\infty} \EE \# \Big\{ X \in \NN \cap \Rect(r,\beta,\eps) : \, A(X) \not= \A(X) \Big\} = 0 ~.
\end{equation}
\end{proposition}

Let us underline that, from Proposition \ref{prop:RPTApprox}, it is not difficult to derive a similar limit to Proposition \ref{prop:EuclFPPtreeStep1} of Section \ref{sect:proofEuclFPPtree} or Proposition \ref{prop:LPPtreeStep1} of Section \ref{sect:proofLPP}. Indeed, let us consider a sequence of local functions $(F_{r})_{r>0}$ such that the stabilizing set $D_{r}$ of $F_{r}$ is equal to the shifted rectangle $\Rect(r,\beta,\eps)-(r,0)$. Then,
\begin{eqnarray*}
d_{TV} \left( F_{r}((r,0),\TT_{\rho}) , F_{r}(O,\FF_{\rho}) \right) & = & d_{TV} \left( F_{r}((r,0),\TT_{\rho}) , F_{r}((r,0),\FF_{\rho}) \right) \\
& \leq & \PP \left( \exists X \in \NN \cap \Rect(r,\beta,\eps) , \, A(X) \not= \A(X) \right) \\
& \leq & \EE \# \Big\{ X \in \NN \cap \Rect(r,\beta,\eps) : \, A(X) \not= \A(X) \Big\} ~,
\end{eqnarray*}
which tends to $0$ as $r\to\infty$ thanks to (\ref{RPTApproxMeanError}).

\begin{proof}{\textit{(of Proposition \ref{prop:RPTApprox})}}
Let us set $n_{r}=\lfloor Cr^{3\beta/2+\eps}\rfloor$ where $C>0$ is a constant chosen large enough such that $\PP(\NN(\Rect(r,\beta,\eps))>n_{r}+k)$ is smaller than $e^{-(n_{r}+k)}$, for any integer $k$ and any large $r$. This is possible by Lemma \ref{lem:Talagrand} since $n_{r}$ is of the same order than the area of $\Rect(r,\beta,\eps)$. Hence, it is not difficult to show that
\begin{eqnarray*}
\EE \left\lbrack \# \Big\{ X \in \NN \cap \Rect(r,\beta,\eps) : \, A(X) \not= \A(X) \Big\} \II_{\NN(\Rect(r,\beta,\eps))>n_{r}} \right\rbrack \hspace*{1cm}\\
\hspace*{1cm}\leq \EE \left\lbrack \NN(\Rect(r,\beta,\eps)) \II_{\NN(\Rect(r,\beta,\eps))>n_{r}} \right\rbrack
\end{eqnarray*}
tends to $0$ as $r\to\infty$. To get (\ref{RPTApproxMeanError}), it then remains to state that
\begin{equation}
\label{RPTApproxMeanError<nr}
\lim_{r\to\infty} \EE \left\lbrack \# \Big\{ X \in \NN \cap \Rect(r,\beta,\eps) : \, A(X) \not= \A(X) \Big\} \II_{\NN(\Rect(r,\beta,\eps))\leq n_{r}} \right\rbrack = 0 ~.
\end{equation}
For any integer $k$,
$$
\EE \left\lbrack \left. \sum_{X \in \NN\cap\Rect(r,\beta,\eps)} \II_{A(X) \not= \A(X)} \; \right| \; \NN (\Rect(r,\beta,\eps)) = k \right\rbrack = \EE \left\lbrack \sum_{i=1}^{k} \ind_{A(X_{i}) \not= \A(X_{i})} \right\rbrack ~,
$$
where $X_{1},\ldots,X_{k}$ are independent random variables uniformly distributed in $\Rect(r,\beta,\eps)$. However, the probability that one of these points admits different ancestors in $\TT_{\rho}$ and $\FF_{\rho}$ is bounded:

\begin{lemma}
\label{lem:approxRPT}
Let $X=xe^{\i\theta}$ and $l\in(\rho,x/2)$. Then, there exist positive constants $c_{0}$, $x_{0}$ and $\theta_{0}$ (only depending on $\rho$) such that for any $x\geq x_{0}$ and $0\leq\theta\leq\theta_{0}$
$$
\PP ( A(X) \not= \A(X) ) \leq e^{-2\rho l} + c_{0} \Big(l^{2}\theta + \frac{1}{x} \Big) ~.
$$
\end{lemma}

Any point $xe^{\i\theta}$ of $\Rect(r,\beta,\eps)$ has an euclidean norm $x$ larger than $r$ and an angle $\theta$ smaller than $\arctan(r^{\beta/2+\eps}/r)$, so smaller than $r^{\beta/2+\eps}/r$. Applying Lemma \ref{lem:approxRPT} with $r$ large enough, we get:
$$
\EE \left\lbrack \left. \sum_{X \in \NN\cap \Rect(r,\beta,\eps)} \II_{A(X) \not= \A(X)} \; \right| \; \NN (\Rect(r,\beta,\eps)) = k \right\rbrack \hspace*{2cm}
$$
$$
\hspace*{7cm} \leq k \left(  e^{-2\rho l} + c_{0} \left( l^{2} \frac{r^{\beta/2+\eps}}{r} + \frac{1}{r} \right) \right) ~.
$$
It then follows:
$$
\EE \left\lbrack \# \Big\{ X \in \NN \cap \Rect(r,\beta,\eps) : \, A(X) \not= \A(X) \Big\} \II_{\NN(\Rect(r,\beta,\eps))\leq n_{r}} \right\rbrack \hspace*{2cm}
$$
$$
\hspace*{7cm} \leq n_{r} \left(  e^{-2\rho l} + c_{0} \left( l^{2} \frac{r^{\beta/2+\eps}}{r} + \frac{1}{r} \right) \right) ~.
$$
Finally, this upper bound tends to $0$ as $r\to\infty$ since $\beta+\eps<1/2$ and taking $l=\ln(r^{\beta'})$ with $\beta'$ large enough.
\end{proof}

The proof of Lemma \ref{lem:approxRPT} is strongly inspired from Theorem 2.4 of \cite{BB}.

\begin{proof}{\textit{(of Lemma \ref{lem:approxRPT})}}
Let $l_{X}$ be the difference of abscissas between $X=xe^{\i\theta}$ and its ancestor $\A(X)$ in the directed forest $\FF_{\rho}$. The horizontal cylinder $U_{X} = X + (-l_{X} ; 0)\!\times\! [-\rho ; \rho]$ avoids the PPP $\NN$: $U_{X}$ admits $\A(X)$ on its west side (see Figure \ref{fig:approxRPT}). Hence, the probability that $l_{X}$ is larger than a given $l$ is smaller than $e^{-2\rho l}$.

\begin{figure}[!ht]
\begin{center}
\psfrag{a}{\small{$O$}}
\psfrag{b}{\small{$\theta$}}
\psfrag{c}{\small{$X=xe^{\i\theta}$}}
\psfrag{d}{\small{$\A(X)$}}
\psfrag{e}{\small{$U_{X}$}}
\psfrag{f}{\small{$\CC_{x}$}}
\includegraphics[width=7.5cm,height=5cm]{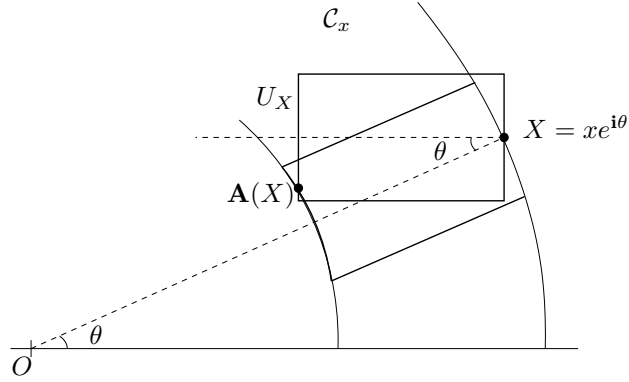}
\end{center}
\caption{\label{fig:approxRPT} Imagine $x$ is large and the angle $\theta$ is small enough so that $X=xe^{\i\theta}$ is very close to the horizontal axis. Then the axis of the cylinder $\textrm{Cyl}(X,\rho)$ is almost horizontal and it becomes difficult for the ancestors $A(X)$ and $\A(X)$ to be different.}
\end{figure}

From now on we assume that $l_{X}\leq l$. So as to the ancestors $A(X)$ and $\A(X)$ differ, two alternatives may be distinguished. Either $\A(X)$ belongs to the set
\begin{equation}
\label{set1}
\Big( X + [-l ; 0)\!\times\! [-\rho ; \rho] \Big) \setminus \textrm{Cyl}(X,\rho) ~,
\end{equation}
either $\A(X)$ belongs to $\textrm{Cyl}(X,\rho)$. This second alternative forces $A(X)$ to be in 
\begin{equation}
\label{set2}
\Big( \textrm{Cyl}(X,\rho) \setminus D(O,|A(X)|) \Big) \setminus U_{X} ~.
\end{equation}
By elementary computations, we check that the area of the sets (\ref{set1}) and (\ref{set2}) is smaller than $c_{0}(l^{2}\theta + 1/x)$ provided $x\geq x_{0}$, $\theta\leq \theta_{0}$ and $\rho\leq l\leq x/2$. The constants $c_{0}$, $x_{0}$ and $\theta_{0}$ only depend on $\rho$.
\end{proof}

Thanks to Proposition \ref{prop:RPTApprox}, we are now able to prove Lemma \ref{lem:RPT-Claim2}.

\begin{proof}{\textit{(of Lemma \ref{lem:RPT-Claim2})}}
Let us first denote by $\mathcal{Z}$ the following event: each edge $(X,A(X))$ of the RPT $\TT_{\rho}$ crossing $I_{r}$ from right to left (i.e. $A(X)(1)<r<X(1)$) satisfies $X\in\Rect(r,\beta,\eps)$. Hence, we write:
\begin{equation}
\label{BigMajorApproxRPT}
\EE \chi_{r}(\alpha,\beta,\eps) \leq \EE \left( \chi_{r}(\alpha,\beta,\eps)-\eta_{r}(\alpha,\beta,\eps) \right) \II_{\mathcal{Z}} + \EE \chi_{r}(\alpha,\beta,\eps)\II_{\mathcal{Z}^{c}} + \EE \eta_{r}(\alpha,\beta,\eps) ~.
\end{equation}
So, in order to obtain Lemma \ref{lem:RPT-Claim2}, we are going to state that
\begin{equation}
\label{BigMajorApproxRPT-1}
\lim_{r\to\infty} \EE \chi_{r}(\alpha,\beta,\eps)\II_{\mathcal{Z}^{c}} = 0 ~,
\end{equation}
thus
\begin{equation}
\label{BigMajorApproxRPT-2}
\mbox{a.s. } \, \left( \chi_{r}(\alpha,\beta,\eps)-\eta_{r}(\alpha,\beta,\eps) \right) \II_{\mathcal{Z}} \leq \# \Big\{ X \in \NN \cap \Rect(r,\beta,\eps) : \, A(X) \not= \A(X) \Big\}
\end{equation}
and then apply Proposition \ref{prop:RPTApprox}.

Let us start with the proof of (\ref{BigMajorApproxRPT-1}). The r.v. $\chi_{r}(\alpha,\beta,\eps)$ is smaller than the number of edges of the RPT crossing the arc, say $\arc'_{r}$, of the circle centered at the origin and passing through both endpoints of $I_{r}$. Let $Y^{(\alpha)}$ this number. Let us cover the (large) arc $\arc'_{r}$ by $\lfloor cr^{\alpha}\rfloor$ nonoverlapping (small) arcs with length $1$ and let us denote by $Y_{i}$ the number of edges of the RPT crossing the $i-$th small arc. The r.v.'s $Y_{i}$ are identically distributed by isotropy and admit a second order moment (see Lemma \ref{lem:2ndmomentRPT} at the end of this section). So, by Cauchy-Schwarz,
$$
\EE \chi_{r}(\alpha,\beta,\eps)^{2} \leq \EE \left( Y^{(\alpha)} \right)^{2} \leq \EE \left( \sum_{1\leq i\leq \lfloor cr^{\alpha}\rfloor} Y_{i} \right)^{2} \leq M r^{2\alpha} ~,
$$
for some $M>0$. Consequently,
$$
\EE \chi_{r}(\alpha,\beta,\eps)\II_{\mathcal{Z}^{c}} \leq M^{1/2} r^{\alpha} \PP(\mathcal{Z}^{c})^{1/2}
$$
which tends to $0$ as $r\to\infty$. Indeed, on the event $\mathcal{Z}^{c}$, there exists an edge crossing $I_{r}$ of length larger than $\min\{r^{\beta/2+\eps}-r^{\alpha},r^{\beta}\}$. Since $\beta/2\geq\alpha$, $\PP(\mathcal{Z}^{c})$ decreases exponentially fast.

Let us prove (\ref{BigMajorApproxRPT-2}). Let $i_{1},\ldots,i_{n}$ be $n=n(\omega)$ different points of $I_{r}$ counted by $\chi_{r}(\alpha,\beta,\eps)$ but not by $\eta_{r}(\alpha,\beta,\eps)$. On the event $\mathcal{Z}$, each point $i_{k}$ is generated by a path (at least) of $\TT_{\rho}$, say $\gamma_{k}$ starting from some Poisson point $X_{k}$, which is included in $\Rect(r,\beta,\eps)$, except its last edge crossing $I_{r}$. Since the intersection point $i_{k}$ is not counted by $\eta_{r}(\alpha,\beta,\eps)$, the path of the forest $\FF_{\rho}$ starting from $X_{k}$, say  $\gamma'_{k}$, leaves $\gamma_{k}$ before crossing $I_{r}$. Precisely, $\gamma_{k}$ and $\gamma'_{k}$ coincide until some Poisson point $Y_{k}\in\Rect(r,\beta,\eps)$ and $A(Y_{k})\not=\A(Y_{k})$. Remark that the bifurcation points $Y_{k}$, $1\leq k\leq n$, are all different. Indeed, $Y_{k}=Y_{l}$ would imply that $\gamma_{k}$ and $\gamma_{l}$ coincide beyond $Y_{k}=Y_{l}$ and then cross $I_{r}$ through the same point: $i_{k}=i_{l}$. Finally, we have exhibited $n$ different vertices of $\Rect(r,\beta,\eps)$ having different ancestors in $\TT_{\rho}$ and $\FF_{\rho}$. (\ref{BigMajorApproxRPT-2}) follows.
\end{proof}

The sequel of the proof only concerns the directed forest $\FF_{\rho}$ and has to state that the supremum limit of $\EE \eta_{r}(\alpha,\beta,\eps)$ is smaller than $1$ as $r\to\infty$. At the end of this section, a conclusion will combine all these intermediate steps and will lead to the expected result, i.e. the convergence to $0$ of $\EE\chi_{r}(0,2r^{\alpha})$.

The next step consists in proving that the paths counted by $\eta_{r}(\alpha,\beta,\eps)$ actually come in  the rectangle $\Rect(r,\beta,\eps)$ from its right side. Let us denote by $\bar{\eta}_{r}(\alpha,\beta,\eps)$ the number of intersection points between $I_{r}$ and paths of $\FF_{\rho}$ crossing the right side of $\Rect(r,\beta,\eps)$, i.e. $J_{r}=\{r+r^{\beta}\}\!\times\![-r^{\beta/2+\eps},r^{\beta/2+\eps}]$.

\begin{lemma}
\label{lem:RPT-Claim3}
Assume $\alpha\leq\beta/2$. The following equality holds:
$$
\limsup_{r\to\infty} \EE \eta_{r}(\alpha,\beta,\eps) = \limsup_{r\to\infty} \EE \bar{\eta}_{r}(\alpha,\beta,\eps) ~.
$$
\end{lemma}

\begin{proof}
Let $0<\eps'<\eps$ and $r$ large enough so that
\begin{equation}
\label{eps'eps}
r^{\alpha}+2r^{\beta/2+\eps'}<r^{\beta/2+\eps} ~.
\end{equation}
Let $A=(r+r^{\beta},r^{\alpha}+r^{\beta/2+\eps'})$. We denote by $\Delta^{A}_{r^{\beta}}$ the maximal deviation of the path $\gamma_{A}^{\infty}$ of $\FF_{\rho}$ w.r.t. the horizontal axis passing by $A$ over the segment $[r,r+r^{\beta}]$:
$$
\Delta^{A}_{r^{\beta}} = \sup \Big\{ |y-A(2)| ; \, (x,y) \in \gamma_{A} \,\mbox{ and }\, r\leq x\leq r+r^{\beta} \Big\} ~.
$$ 
In the same way, we consider the quantity $\Delta^{B}_{r^{\beta}}$ for $B=(r+r^{\beta},-r^{\alpha}-r^{\beta/2+\eps'})$. On the event
$$
\mbox{Dev} = \big\{ \max\{\Delta^{A}_{r^{\beta}} , \Delta^{B}_{r^{\beta}} \} < r^{\beta/2+\eps'} \big\} ~,
$$
the paths $\gamma_{A}^{\infty}$ and $\gamma_{B}^{\infty}$ do not cross the vertical segment $I_{r}$. Moreover, $\gamma_{A}^{\infty}$ and $\gamma_{B}^{\infty}$ are totally included in the rectangle $\Rect(r,\beta,\eps)$ before crossing $\{0\}\!\times\!\R$, thanks to (\ref{eps'eps}). So, on the event $\mbox{Dev}$, any intersection point counted by $\eta_{r}(\alpha,\beta,\eps)$ is produced by a path crossing $J_{r}$. In other words, $\eta_{r}(\alpha,\beta,\eps)\II_{\mbox{Dev}}$ is a.s. smaller than $\bar{\eta}_{r}(\alpha,\beta,\eps)$.

It then remains to show that $\EE\eta_{r}(\alpha,\beta,\eps)\II_{\mbox{Dev}}$ tends to $0$. We can proceed as in the proof of Lemma \ref{lem:RPT-Claim2}. The translation invariance property of the directed forest $\FF_{\rho}$ and the Cauchy-Schwarz inequality allow us to write:
$$
\EE \eta_{r}(\alpha,\beta,\eps)\II_{\mbox{Dev}} \leq C r^{\alpha} \PP(\mbox{Dev}^{c})^{1/2} ~,
$$
for some positive constant $C$. It is then enough to show that:

\begin{lemma}
\label{Dev-Claim3}
For any $\eps''>0$ and any integer $m$, $\PP(\Delta_{r}>r^{1/2+\eps''})$ is a $\mathcal{O}(r^{-m})$, where $\Delta_{r}$ denotes the maximal deviation of the path $\gamma_{(r,0)}^{\infty}$ of $\FF_{\rho}$ w.r.t. the horizontal axis over the segment $[0,r]$.
\end{lemma}

Indeed, replacing $r$ with $r^{\beta}$ and using the translation invariance property of $\FF_{\rho}$, Lemma \ref{Dev-Claim3} says that $\PP(\mbox{Dev}^{c})$ is a $\mathcal{O}(r^{-m \beta})$, for any $m$. This achieves the proof.
\end{proof}

\begin{proof}{\textit{(of Lemma \ref{Dev-Claim3})}}
Let $\eps''>0$ and $m\in\mathbb{N}$. Let us write $\gamma_{(r,0)}^{\infty}=(X_{n})_{n\geq 0}$ the sequence of successive ancestors of $(r,0)=X_{0}$, and for any index $n\geq 1$, let $(Y_{n},Z_{n})$ be the cartesian coordinates of $X_{n}-X_{n-1}$. Now, if we denote by $n(r)$ the first integer $n$ such that $X_{n}(1)$ is negative, the maximal deviation $\Delta_{r}$ can be expressed as
$$
\Delta_{r} = \max_{1\leq k\leq n(r)} \Big| \sum_{i=1}^{k} Z_{i} \Big| ~.
$$
Then,
\begin{equation}
\label{Delta(r,0)Rect}
\PP \Big( \Delta_{r} > r^{1/2+\eps''} \Big) \leq \PP (n(r) > \lfloor cr \rfloor) + \PP \Big( \max_{1\leq k\leq \lfloor cr \rfloor} \Big| \sum_{i=1}^{k} Z_{i} \Big| > r^{1/2+\eps''} \Big) ~.
\end{equation}
The inequality $n(r)>\lfloor cr\rfloor$ means that the partial sum $\sum_{1}^{\lfloor cr\rfloor} Y_{i}$ does not exceed $r$. As the $Y_{i}$'s are i.i.d. r.v.'s with an exponential law as common distribution, Lemma \ref{lem:Gut} applies. Provided the additional parameter $c$ is larger than $(\EE Y_{1})^{-1}$, the quantity $\PP(n(r)>\lfloor cr\rfloor)$ is a $\mathcal{O}(r^{-m})$.

To treat the second term of the r.h.s. of (\ref{Delta(r,0)Rect}), it suffices to apply Lemma \ref{lem:Youri} since the $Z_{i}$'s are i.i.d. according to the uniform distribution on $[-\rho,\rho]$.\end{proof}

This section ends with the following limit which requires the hypothesis $\alpha<\beta/2$ in a crucial way and an estimate of the coalescence time between two paths of $\FF_{\rho}$ stated in \cite{FFW}.

\begin{lemma}
\label{lem:RPT-Claim4}
Assume $\alpha<\beta/2$. Then,
$$
\limsup_{r\to\infty} \EE \bar{\eta}_{r}(\alpha,\beta,\eps) \leq 1 ~.
$$
\end{lemma}

\begin{proof}
Let us consider the r.v. $\tilde{\eta}_{r}(\alpha,\beta)$ defined as the number of intersection points between the vertical axis $\{r\}\times\R$ and paths of $\FF_{\rho}$ crossing the segment $\{r+r^{\beta}\}\!\times\![-r^{\alpha},r^{\alpha}]$. The translation invariance property of $\FF_{\rho}$ provides:
\begin{equation}
\label{Claim4-TransInv}
\EE \bar{\eta}_{r}(\alpha,\beta,\eps) \leq \EE \tilde{\eta}_{r}(\alpha,\beta) ~.
\end{equation}
Indeed, if $I$ and $J$ are two segments resp. included in $\{r\}\times\R$ and $\{r+r^{\beta}\}\times\R$ then $\eta_{r}(I,J)$ denotes the number of intersection points between $I$ and paths of $\FF_{\rho}$ which also cross the segment $J$. Then,
\begin{eqnarray*}
\EE \bar{\eta}_{r}(\alpha,\beta,\eps) & \leq & \EE \eta_{r}(\{r\}\!\times\![-r^{\alpha} , r^{\alpha}],\{r+r^{\beta}\}\!\times\!\R) \\
& \leq & \sum_{k\in\mathbb{Z}} \EE \eta_{r}(\{r\}\!\times\![-r^{\alpha},r^{\alpha}] , \{r+r^{\beta}\}\!\times\![(2k-1)r^{\alpha},(2k+1)r^{\alpha}]) \\
& = & \sum_{k\in\mathbb{Z}} \EE \eta_{r}(\{r\}\!\times\![(-2k-1)r^{\alpha},(-2k+1)r^{\alpha}] , \{r+r^{\beta}\}\!\times\![-r^{\alpha},r^{\alpha}]) \\
& = & \EE \tilde{\eta}_{r}(\alpha,\beta) ~.
\end{eqnarray*}

Let $X\in\NN$ and let $\gamma_{X}^{\infty}=(X_{n})_{n\geq 0}$ the sequence of successive ancestors of $X_{0}=X$. As in Section 2 of \cite{FFW}, we introduce a continuous time Markov process $\gamma_{X}^{\ast}=\{\gamma_{X}^{\ast}(t), t\leq X(1)\}$ associated to the sequence $(X_{n})_{n\geq 0}$ and defined by:
\begin{equation}
\label{paths-ast}
\gamma_{X}^{\ast}(t) = X_{n}(2) , \, \mbox{ for any $t$ such that $X_{n+1}(1) < t \leq X_{n}(1)$.}
\end{equation}
Then, $\tilde{\eta}_{r}(\alpha,\beta)$ is a.s. smaller than $\eta_{r}^{\ast}(\alpha,\beta)$ which counts the intersection points between the vertical axis $\{r\}\times\R$ and paths of $\{\gamma_{X}^{\ast}, X\in\NN\}$ crossing the segment $\{r+r^{\beta}\}\!\times\![-r^{\alpha}-\rho,r^{\alpha}+\rho]$.

To obtain Lemma \ref{lem:RPT-Claim4}, we are going to prove that the supremum limit of $\EE \eta_{r}^{\ast}(\alpha,\beta)$ is smaller than $1$ as $r\to\infty$. Actually, the passage from the forest $\FF_{\rho}$ to the collection $\{\gamma_{X}^{\ast}, X\in\NN\}$ acts as a discretization argument. Indeed, by construction, two paths $\gamma_{X}^{\ast}$ and $\gamma_{Y}^{\ast}$ crossing the axis $\{r+r^{\beta}\}\!\times\!\R$ on two different ordinates $x'$ and $y'$ satisfy a.s. $|x'-y'|>\rho$. We use this remark as follows. Let us first consider $\kappa_{r}$ elements of the vertical line $\{r+r^{\beta}\}\!\times\!\R$, say $x_{1},\ldots,x_{\kappa_{r}}$, such that $|x_{i+1}-x_{i}|=2\rho$ and $[x_{1};x_{\kappa_{r}}]$ contains $\{r+r^{\beta}\}\!\times\![-r^{\alpha}-\rho,r^{\alpha}+\rho]$. One can choose $\kappa_{r}\leq 3r^{\alpha}$. Thus, for any $i=1,\ldots,\kappa_{r}$, $\gamma_{i}^{\ast}$ is the path starting at $x_{i}$ and defined as in (\ref{paths-ast}). As a consequence, with probability $1$,
$$
\eta_{r}^{\ast}(\alpha,\beta) \leq \sum_{i=1}^{\kappa_{r}-1} \II_{\gamma_{i}^{\ast}(r)\not=\gamma_{i+1}^{\ast}(r)} + 1 ~.
$$
In the above upperbound, the term ``$+1$'' is inevitable since any two consecutive paths $\gamma_{X}^{\ast}$ and $\gamma_{Y}^{\ast}$ counted by $\eta_{r}^{\ast}(\alpha,\beta)$ are separated by one couple $(i,i+1)$ such that $\gamma_{i}^{\ast}(r)\not=\gamma_{i+1}^{\ast}(r)$. Now, the event $\{\gamma_{1}^{\ast}(r)\not=\gamma_{2}^{\ast}(r)\}$ means that the coalescence time of the two paths $\gamma_{1}^{\ast}$ and $\gamma_{2}^{\ast}$ is larger than $r^{\beta}$. Thanks to Lemma 2.10 of \cite{FFW}, its probability is bounded by $c_{2}/r^{\beta/2}$ where the constant $c_{2}>0$ only depends on $\rho$. It follows,
$$
\EE \eta_{r}^{\ast}(\alpha,\beta) \leq \kappa_{r}\PP \big( \gamma_{1}^{\ast}(r)\not=\gamma_{2}^{\ast}(r) \big) + 1 \leq 3r^{\alpha} \frac{c_{2}}{r^{\beta/2}} + 1
$$
which tends to $1$ as $r\to\infty$ since $\alpha<\beta/2$.
\end{proof}

We can now conclude. Combining Lemmas \ref{lem:RPT-Claim1}, \ref{lem:RPT-Claim2}, \ref{lem:RPT-Claim3} and \ref{lem:RPT-Claim4}, we obtain that the supremum limit of $\EE\chi_{r}(0,2r^{\alpha})$, say $c(\alpha)$, is smaller than $1$, for any $0<\alpha<1/4$. Let $M>0$ and $\alpha<\alpha'<1/4$. By isotropy, for $r$ large enough,
$$
\EE\chi_{r}(0,2r^{\alpha'}) \geq M \EE\chi_{r}(0,2r^{\alpha}) ~.
$$
Taking supremum limits, it follows that $1\geq M c(\alpha)$. When $M\to\infty$ this forces $c(\alpha)=0$.\\

This section ends with the proof of the following technical result.

\begin{lemma}
\label{lem:2ndmomentRPT}
Let $\psi_{r}$ be the number of edges of the RPT $\TT_{\rho}$ crossing the arc $\arc_{r}(0,l)$ of $\CC_{r}$ ($l>0$ does not depend on $r$). Then,
$$
\limsup_{r\to\infty} \EE \big[ \psi_{r}^{2} \big] < \infty ~.
$$
\end{lemma}

\begin{proof}
The proof is very close to the one of STEP 4 in Section \ref{sect:proofEuclFPPtree} about the Euclidean FPP Trees. Let us give some details in the context of the RPT. It is enough to prove that $\PP(\psi_{r}>n)$ decreases exponentially fast with $n$ and uniformly on $r$.

By Lemma \ref{lem:Talagrand}, we first control the number of vertices in the disk $D((r,0),R)$:
\begin{equation}
\label{lastlemmaRPT-1}
\PP ( \NN (D((r,0),R)) > n ) \leq e^{-n \ln \left( \frac{n}{e\pi R^{2}} \right)} ~,
\end{equation}
for any $R>0$. Now, the conjunction of $\psi_{r}>n$ and $\NN(D((r,0),R))\leq n$ implies the existence of a vertex $X$ outside $D((r,0),R)$ whose edge $[A(X) ; X]$ crosses the arc $\arc_{r}(0,l)$. Then, the random cylinder $\textrm{Cyl}(X,\rho)^{\ast}$ avoids the PPP $\NN$. Such situation should occur with small probability. Indeed, let us consider a family of $k(R)\leq 2\pi R/\rho$ deterministic rectangles of size $\rho\times R/3$ included in $D((r,0),R)\setminus D((r,0),R/2)$ and such that one of them is including in $\textrm{Cyl}(X,\rho)^{\ast}$. This selected rectangle avoids the PPP $\NN$:
\begin{equation}
\label{lastlemmaRPT-2}
\PP( \psi_{r} > n \; \mbox{ and } \; \NN (D((r,0),R)) \leq n ) \leq \frac{2\pi R}{\rho} e^{-\rho R/3} ~.
\end{equation}
To conclude, it suffices to take $R=n^{1/4}$ in the bounds given in (\ref{lastlemmaRPT-1}) and (\ref{lastlemmaRPT-2}). Remark also these two bounds do not depend on $r$.
\end{proof}

\section{Almost sure convergence for the RPT}
\label{sect:asRPT}

The goal of this section is to prove that $\chi_{r}/r^{\alpha}$ tends to $0$ with probability $1$ for any $\alpha>3/4$.\\

We first use Theorem \ref{theo:RPTstraight} of Section \ref{sect:RPTstraight} which asserts that, for any $\eps>0$, the event $A_{r}^{\eps}$ defined below has a probability tending to $1$ as $r\to\infty$.
$$
A_{r}^{\eps} = \left\{ \forall X \in \NN\cap D(O,r)^{c} , \, \cup_{Y\in \TT^{\mbox{\tiny{out}}}_{\rho}(X)} \textrm{Cyl}(Y,\rho)^{\ast} \subset C(X,|X|^{-\frac{1}{2}+\eps}) \right\} ~,
$$
where $\TT^{\mbox{\tiny{out}}}_{\rho}(X)$ is the subtree of the RPT $\TT_{\rho}$ rooted at $X$ and $C(X,|X|^{-\frac{1}{2}+\eps})$ is the semi-infinite cone with axis $X$ and opening angle $|X|^{-\frac{1}{2}+\eps}$.

Let $\eps>0$ and $r_{0}$ such that $\PP(A_{r_{0}}^{\eps})\geq 1/2$. For $r\geq r_{0}$, we split the circle $\CC_{r}$ into $r$ nonoverlapping arcs $\arc_{1},\ldots,\arc_{r}$ with length $2\pi$. Precisely, $\arc_{i}=\arc_{r}(\theta,2\pi)$. The number of semi-infinite paths of $\TT_{\rho}$ crossing $\arc_{i}$ is denoted by $X_{i}^{(r)}$. By isotropy of $\TT_{\rho}$, the $X_{i}^{(r)}$'s are identically distributed and satisfy $\chi_{r}=\sum_{i=1}^{r}X_{i}^{(r)}$. We also set
$$
S_{r} = \sum_{i=1}^{r} Y_{i}^{(r)} \, \mbox{ where } \, Y_{i}^{(r)} = X_{i}^{(r)} - \EE \big[ X_{i}^{(r)} \,|\, A_{r_{0}}^{\eps} \big] ~.
$$
Remark that the event $A_{r_{0}}^{\eps}$ is preserved under rotations. So, the r.v. $Y_{i}^{(r)}$'s are also identically distributed.

The next lemma reduces the proof of $\PP(r^{-\alpha}\chi_{r}\to 0)=1$ to: for any $r_{0}$ large enough,
\begin{equation}
\label{CVasAr0}
\PP \left( r^{-\alpha} S_{r} \to 0 \,|\, A_{r_{0}}^{\eps} \right) = 1 ~.
\end{equation}

\begin{lemma}
\label{lem:as-RPT-step1}
If $S_{r}/r^{\alpha}$ a.s. tends to $0$ as $r\to\infty$ conditionally to $A_{r_{0}}^{\eps}$, for any $r_{0}$ large enough, then $\chi_{r}/r^{\alpha}$ a.s. tends to $0$ as $r\to\infty$ too.
\end{lemma}

\begin{proof}
Let us write
$$
\frac{\chi_{r}}{r^{\alpha}} = \frac{S_{r}}{r^{\alpha}} + \frac{1}{r^{\alpha}} \EE \big[ \chi_{r} \,|\, A_{r_{0}}^{\eps} \big] ~.
$$
By the first part of Theorem \ref{theo:sublinRPT}, $r^{-\alpha} \EE[\chi_{r} | A_{r_{0}}^{\eps}]\leq 2r^{-\alpha} \EE[\chi_{r}]$ which tends to $0$ as $r\to\infty$. So the almost sure convergence of $S_{r}/r^{\alpha}$ to $0$ conditionally to $A_{r_{0}}^{\eps}$ implies the one of $\chi_{r}/r^{\alpha}$ still to $0$ and conditionally to $A_{r_{0}}^{\eps}$. Since $\PP(A_{r_{0}}^{\eps})\to 1$ as $r_{0}\to\infty$, the (unconditionally) a.s. convergence of $\chi_{r}/r^{\alpha}$ to $0$ follows.
\end{proof}

The next estimate is the key ingredient to obtain (\ref{CVasAr0}).

\begin{lemma}
\label{lem:as-RPT-step-key}
There exists $C>0$ such that for any $r\geq r_{0}$,
$$
\EE \big[ S_{r}^{2} \,|\, A_{r_{0}}^{\eps} \big] \leq C r^{3/2+\eps} ~.
$$
\end{lemma}

The proof of (\ref{CVasAr0}) is a consequence of Lemma \ref{lem:as-RPT-step-key} and the Borel-Cantelli lemma. Let $l>1$, $\delta>0$,
\begin{eqnarray*}
\PP \left( r^{-\alpha l} S_{r^{l}} > \delta \,|\, A_{r_{0}}^{\eps} \right) & \leq & \delta^{-2} r^{-2\alpha l} \EE \big[ S_{r^{l}}^{2} \,|\, A_{r_{0}}^{\eps} \big] \\
& \leq & \delta^{-2} C r^{l(-2\alpha +3/2+\eps)} 
\end{eqnarray*}
which is the general term of a convergent series. Indeed, $\eps>0$ can be chosen small enough so that $-2\alpha +3/2+\eps<0$ and $l$ large enough so that $l(-2\alpha +3/2+\eps)<-1$. Thus, the Borel-Cantelli lemma gives the a.s. convergence of $S_{r^{l}}/r^{\alpha l}$ to $0$ conditionally to $A_{r_{0}}^{\eps}$. Then, the a.s. convergence of $S_{r}/r^{\alpha}$ to $0$ (conditionally to $A_{r_{0}}^{\eps}$) easily follows. Given $r$, we consider the integer $n=n(r)$ such that $(n-1)^{l}\leq r<n^{l}$. Since the sequence $(S_{r})_{r>0}$ is nondecreasing a.s. we can write:
$$
0 \leq \frac{S_{r}}{r^{\alpha}} \leq \frac{S_{n^{l}}}{(n-1)^{\alpha l}} \leq \frac{S_{n^{l}}}{n^{\alpha l}} \frac{n^{\alpha l}}{(n-1)^{\alpha l}}
$$ 
which, conditionally to $A_{r_{0}}^{\eps}$, tends to $0$ as $r\to\infty$ with probability $1$.\\

So, it only remains to state Lemma \ref{lem:as-RPT-step-key} whose proof is based on the conditional independence (w.r.t. the event $A_{r_{0}}^{\eps}$) between $Y_{i}^{(r)}$ and $Y_{j}^{(r)}$ provided the difference $i-j$ is large enough.

\begin{proof}{\textit{(of Lemma \ref{lem:as-RPT-step-key})}}
Let $r\geq r_{0}$. Let us first expand the considered expectation:
\begin{equation}
\label{as-RPT-expand}
\EE \big[ S_{r}^{2} \,|\, A_{r_{0}}^{\eps} \big] = \sum_{i=1}^{r} \EE \big[ (Y_{i}^{(r)})^{2} \,|\, A_{r_{0}}^{\eps} \big] + \sum_{i\not= j} \EE \big[ Y_{i}^{(r)} Y_{j}^{(r)} \,|\, A_{r_{0}}^{\eps} \big] ~.
\end{equation}
The first term of the r.h.s. of (\ref{as-RPT-expand}) is bounded using isotropy, Lemma \ref{lem:2ndmomentRPT} and $\PP(A_{r_{0}}^{\eps})\geq 1/2$:
\begin{eqnarray*}
\sum_{i=1}^{r} \EE \big[ (Y_{i}^{(r)})^{2} \,|\, A_{r_{0}}^{\eps} \big] & = & r \EE \big[ (Y_{1}^{(r)})^{2} \,|\, A_{r_{0}}^{\eps} \big] \\
& \leq & r \EE \big[ (X_{1}^{(r)})^{2} \,|\, A_{r_{0}}^{\eps} \big] \\
& \leq & 2 r \EE \big[ (X_{1}^{(r)})^{2} \big] \\
& \leq & 2 M r 
\end{eqnarray*}
for some positive $M>0$. Let us now focus on the second term of the r.h.s. of (\ref{as-RPT-expand}):
$$
\sum_{i\not= j} \EE \big[ Y_{i}^{(r)} Y_{j}^{(r)} \,|\, A_{r_{0}}^{\eps} \big] = r \sum_{2\leq i\leq r} \EE \big[ Y_{1}^{(r)} Y_{i}^{(r)} \,|\, A_{r_{0}}^{\eps} \big] ~.
$$
Here is the reason why we have conditioned by $A_{r_{0}}^{\eps}$. On this event, the r.v. $X_{i}^{(r)}$ (so does $Y_{i}^{(r)}$) only depends on the PPP $\NN$ restricted to the semi-infinite cone with apex the origin and whose intersection with $\CC_{r}$ is the arc centered with $\arc_{i}$ and with length $2(\pi+\rho+r^{1/2+\eps})$. Then $Y_{1}^{(r)}$ and $Y_{i}^{(r)}$ are independent conditionally to $A_{r_{0}}^{\eps}$ provided the difference $i-1$ (taken modulo $r$) is larger than $2(\pi+\rho+r^{1/2+\eps})/2\pi$. When this is the case,
$$
\EE \big[ Y_{1}^{(r)} Y_{i}^{(r)} \,|\, A_{r_{0}}^{\eps} \big] = \EE \big[ Y_{1}^{(r)} \,|\, A_{r_{0}}^{\eps} \big] \EE \big[ Y_{i}^{(r)} \,|\, A_{r_{0}}^{\eps} \big] = 0 ~.
$$
Otherwise, by isotropy and the Cauchy-Schwarz inequality,
$$
\EE \big[ Y_{1}^{(r)} Y_{i}^{(r)} \,|\, A_{r_{0}}^{\eps} \big] \leq \EE \big[ (Y_{1}^{(r)})^{2} \,|\, A_{r_{0}}^{\eps} \big] \leq 2M
$$
as previously. As consclusion, there exists a positive constant $\kappa=\kappa(\rho)$ such that the r.h.s. of (\ref{as-RPT-expand}) is bounded from above by $2\kappa M r^{3/2+\eps}$. This achieves the proof of Lemma \ref{lem:as-RPT-step-key}.
\end{proof}

\section{The Radial Poisson Tree is straight}
\label{sect:RPTstraight}

For any vertex $X\in\NN$, we denote by $\TT^{\mbox{\tiny{out}}}_{\rho}(X)$ the subtree of the RPT $\TT_{\rho}$ rooted at $X$; $\TT^{\mbox{\tiny{out}}}_{\rho}(X)$ is the collection of paths of $\TT_{\rho}$ from $O$ to $X'\in\NN$, passing by $X$, whose common part from $O$ to $X$ has been deleted. Let $C(X,\alpha)$ for nonzero $X\in\R^{2}$ and $\alpha\geq 0$ be the cone $C(X,\alpha)=\{Y\in\R^{2}, \theta(X,Y)\leq\alpha\}$ where $\theta(X,Y)$ is the absolute value of the angle (in $[0;\pi]$) between $X$ and $Y$. Theorem \ref{theo:RPTstraight} means the subtrees $\TT^{\mbox{\tiny{out}}}_{\rho}(X)$ are becoming thinner as $|X|$ increases.

\begin{theorem}
\label{theo:RPTstraight}
With probability $1$, the Radial Poisson Tree $\TT_{\rho}$ is straight. Precisely, with probability $1$ and for all $\eps>0$, the subtree $\TT^{\mbox{\tiny{out}}}_{\rho}(X)$ is included in the cone $C(X,|X|^{-\frac{1}{2}+\eps})$ for all but finitely many $X\in\NN$.
\end{theorem}

For a real number $r>0$, let us denote by $\gamma_{r}$ the path of the RPT from $X_{0}=(r,0)$ to $O$. It can be described by the sequence of successive ancestors of $X_{0}$, say $X_{0},X_{1},\ldots,X_{h(r)}=O$ where $h(r)$ denotes the number of steps to reach the origin. Let us denote by $\Delta_{r}$ the maximal deviation of $\gamma_{r}$ w.r.t the horizontal axis:
$$
\Delta_{r} = \max_{0\leq k\leq h(r)} |X_{k}(2)|
$$
(where $X_{k}(2)$ is the ordinate of $X_{k}$).

\begin{proposition}
\label{prop:RPTdeviation}
The following holds for all $\eps>0$ and all $n\in\N$,
\begin{equation}
\label{RPTdeviation}
\PP ( \Delta_{r} \geq r^{\frac{1}{2}+\eps} ) = \mathcal{O}(r^{-n}) ~.
\end{equation}
\end{proposition}

Theorem \ref{theo:RPTstraight} is a consequence of Proposition \ref{prop:RPTdeviation}. Indeed, (\ref{RPTdeviation}) implies that with high probability the path $\gamma_{r}$ remains inside the cone $C((r,0),f(r))$ with $f(r)=r^{\frac{1}{2}+\eps}/r$. We then conclude by isotropy of the RPT $\TT_{\rho}$.

\begin{proof}{\textit{(of Theorem \ref{theo:RPTstraight})}}
We first show that the number of vertices $X\in\NN$ whose deviation of the path from $X$ to $O$ w.r.t the axis $(OX)$ is larger than $|X|^{\frac{1}{2}+\eps}$ is a.s. finite. To do it, we can follow the proof of Theorem 5.4 of \cite{BB} and use Proposition \ref{prop:RPTdeviation}, the isotropic character of $\TT_{\rho}$ and the Campbell's formula. Thus, thanks to the Borel-Cantelli lemma, we prove that a.s. all but finitely many $X\in\NN$ satisfy $|X-A(X)|\leq |X|^{\frac{1}{2}}$. To conclude, it suffices to apply Lemma 2.7 of \cite{HN01} replacing $\frac{3}{4}$ with $\frac{1}{2}$.
\end{proof}

The rest of this section is devoted to the proof of Proposition \ref{prop:RPTdeviation}. Baccelli and Bordenave have proved in \cite{BB} the same result about the RST. They first bound the fluctuations of the radial path $\gamma_{r}$ by the ones of a directed path (which actually belongs to the DSF with direction $-(1,0)$), and then compute its fluctuations. The main difficulty of their work lies in the second step. The main obstacle here consists in comparing the radial path $\gamma_{r}$ to a directed one having good properties. Indeed, one easily observes that the ancestor of a given point $X$ (with $X(2)>0$) for the RPT $\TT_{\rho}$ may be above the ancestor of the same point but for the directed forest $\FF_{\rho}$: see Figure \ref{fig:approxRPT}.

As in \cite{BB}, let us start with introducing the path $\gamma^{+}_{r}$ of the RPT defined on $\NN\cap(\R\!\times\!\R^{+})$ starting at $X_{0}=(r,0)$ and ending at $O$. It is not difficult to see that, built on the same PPP $\NN$, $\gamma^{+}_{r}$ is above $\gamma_{r}$. Considering the same path $\gamma^{-}_{r}$ but this time defined on $\NN\cap(\R\!\times\!\R^{-})$ allows to trapp $\gamma_{r}$ between $\gamma^{+}_{r}$ and $\gamma^{-}_{r}$. By symmetry, it follows:
$$
\forall t > 0 , \; \PP ( \Delta_{r} \geq t ) \leq 2 \PP ( \Delta^{+}_{r} \geq t )
$$
where $\Delta^{+}_{r}$ denotes the maximal deviation of $\gamma^{+}_{r}$ w.r.t. the axis $(OX_{0})$.\\

From now on, we only consider the PPP $\NN\cap(\R\!\times\!\R^{+})$. We still denote by $A(X)$ the ancestor of $X$ using $\NN\cap(\R\!\times\!\R^{+})$ and by $(X_{0},X_{1},X_{2}\ldots)$ the sequence of successive ancestors of $\gamma^{+}_{r}$. In order to bound its fluctuations, a natural way to proceed would be to consider the  vectors $U_{n+1}=(X_{n+1}-X_{n})e^{-\i \mbox{arg}(X_{n})}$, for $n\geq 0$. Nevertheless, it is not clear how to compare the fluctuations of the sequence $(U_{1},U_{1}+U_{2}\ldots)$ with the ones of $\gamma^{+}_{r}$. In particular, the inequality $X_{n}(2)\leq U_{1}(2)+\ldots+U_{n}(2)$ does not hold (for instance, when $n=2$ and $X_{2}(1)$ larger than $X_{1}(1)$). This is the reason why the following construction is considered.

Let $X\in(\R_{+})^{2}$. Let $V_{X}^{-}$ be the set of points of $\textrm{Cyl}(X,\rho)$ whose abscissas are larger than $X(1)$, and let $V_{X}^{+}$ be its image by the reflection w.r.t. the axis $(OX)$. See Figure \ref{fig:starancestor}. We can then define the $\ast$-ancestor of $X$ as the element $A^{\ast}(X)$ of $(\R_{+})^{2}$, but not necessarily of $\NN$, satisfying $|A^{\ast}(X)|=|A(X)|$ and:
\begin{enumerate}
\item If $A(X)\notin V_{X}^{-}\cup V_{X}^{+}$ then $A^{\ast}(X)(2)=X(2)\pm d(A(X),(OX))$ with the symbol $+$ when $A(X)$ is above the line $(OX)$, and $-$ when $A(X)$ is below $(OX)$.
\item If $A(X)\in V_{X}^{-}$ then $A^{\ast}(X)(2)=X(2)-\min\{\rho,d(A(X),\{Y, Y(2)=X(2)\})\}$.
\item If $A(X)\in V_{X}^{+}$ then $A^{\ast}(X)(2)=X(2)+\min\{\rho,d(\mbox{Sym} A(X),\{Y, Y(2)=X(2)\})\}$ where $\mbox{Sym} A(X)$ denotes the image of $A(X)$ by the reflection w.r.t. the axis $(OX)$.
\end{enumerate}
Let us give the motivations for this construction. Leaving out the fact that there is no point of the PPP below the horizontal axis, the construction of the case $1$ ensures that the distribution of the random variable $A^{\ast}(X)(2)-X(2)$ is symmetric on $[-\rho;\rho]$ and $A^{\ast}(X)(2)\geq A(X)(2)$ (see Lemma \ref{lem:starancestor}). When $A(X)(1)\geq X(1)$ (case $2$) we have to proceed differently to ensure that $A^{\ast}(X)(2)\geq A(X)(2)$. The case $3$ is introduced so as to conserve a symmetric construction w.r.t. the line $(OX)$.

Finally, remark that the construction of the $\ast$-ancestor of $X$ is possible provided that
\begin{equation}
\label{conditionAast}
X(1)\geq 0 \; \mbox{ and  } \; X(2) + \rho \leq |A(X)| ~.
\end{equation}

\begin{figure}[!ht]
\begin{center}
\psfrag{a}{\small{$O$}}
\psfrag{b}{\small{$X$}}
\psfrag{c}{\small{$A(X)$}}
\psfrag{d}{\small{$A^{\ast}(X)$}}
\includegraphics[width=10cm,height=7cm]{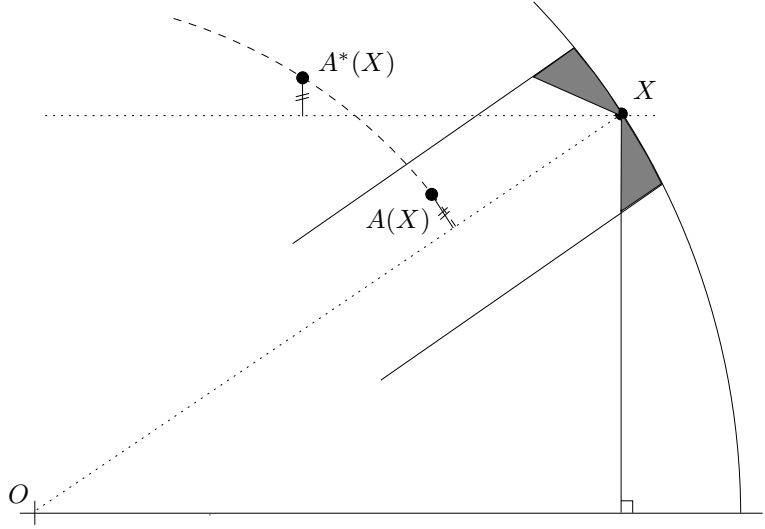}
\end{center}
\caption{\label{fig:starancestor} The gray sets are $V_{X}^{-}$ and $V_{X}^{+}$ which are symmetric from each other w.r.t. $(OX)$. On this picture, the ancestor $A(X)$ does not belong to $V_{X}^{-}\cup V_{X}^{+}$: this is the case $1$. The distance and the relative position between the $\ast$-ancestor $A^{\ast}(X)$ and the horizontal line $\{Y, Y(2)=X(2)\}$ are the same than between $A(X)$ and $(OX)$. Note also that $A(X)$ and $A^{\ast}(X)$ are on the same circle (centered at the origin).}
\end{figure}

\begin{lemma}
\label{lem:starancestor}
Let $X\in(\R_{+})^{2}$. Assume that $X$ and its ancestor $A(X)$ built on the PPP $\NN\cap(\R\!\times\!\R^{+})$ satisfy (\ref{conditionAast}). Then, $A^{\ast}(X)$ is well defined, $|A^{\ast}(X)|=|A(X)|$ and $A^{\ast}(X)(2)\geq A(X)(2)$. Moreover, conditionally to $\textrm{Cyl}(X,\rho)^{\ast}\subset\{Y, Y(2)\geq 0\}$, the distribution of the random variable $A^{\ast}(X)(2)-X(2)$ is symmetric on $[-\rho;\rho]$.
\end{lemma}

\begin{proof}
Conditionally to $\textrm{Cyl}(X,\rho)^{\ast}\subset\{Y, Y(2)\geq 0\}$, the distance between $A(X)$ and the line $(OX)$ is a random variable whose distribution is symmetric on $[-\rho;\rho]$. By construction of the $\ast$-ancestor $A^{\ast}(X)$, the same holds for $A^{\ast}(X)(2)-X(2)$.

It remains to prove that $A^{\ast}(X)(2)\geq A(X)(2)$. In the case $1$, the $\ast$-ancestor $A^{\ast}(X)$ has the same ordinate than $A(X)e^{-\i \mbox{arg}(X)}$, which is larger than $A(X)(2)$ since $A(X)\notin V_{X}^{-}$. In the case $2$, it suffices to write
$$
X(2) - A^{\ast}(X)(2) \leq d(A(X),\{Y, Y(2)=X(2)\}) = X(2) - A(X)(2) ~.
$$
Besides, the case $2$ is the only way to have $A^{\ast}(X)=A(X)\in\NN$. The case $3$ requires more details. Let $\mathcal{D}$ be the horizontal line passing by $X$, let $\mathcal{D}'$ be the tangent line touching the circle $\CC_{|X|}$ at $X$, and let $\mathcal{D}''$ be the image of $\mathcal{D}$ by the reflection w.r.t. the axis $(OX)$. Geometrical arguments show that $\mathcal{D}'$ actually is the line bisector of $\mathcal{D}$ and $\mathcal{D}''$. As a consequence, $\mbox{Sym} A(X)$ is closer to $\mathcal{D}''$ than to $\mathcal{D}$:
$$
d(A(X),\mathcal{D}) = d(\mbox{Sym} A(X),\mathcal{D}'') \leq d(\mbox{Sym} A(X),\mathcal{D}) ~.
$$
Using $A(X)\in V_{X}^{+}$ and $d(V_{X}^{+},\mathcal{D})\leq\rho$, we can then conclude:
\begin{eqnarray*}
A^{\ast}(X)(2) - X(2) & = & \min\{\rho , d(\mbox{Sym} A(X),\mathcal{D})\} \\
& \geq & \min\{\rho , d(A(X),\mathcal{D})\} \\
& = & d(A(X),\mathcal{D}) \\
& = & A(X)(2) - X(2) ~.
\end{eqnarray*}
\end{proof}

Now, we need to control that with high probability the ancestor $A(X)$ (build on the PPP $\NN\cap(\R\!\times\!\R^{+})$) is not too far from $X$. Let $\Omega(r,\kappa)=\{\forall X \in (\R_{+})^{2}\cap D(O,r) , \, |X|-|A(X)|\leq \kappa \}$.

\begin{lemma}
\label{lem:Omega}
Let $n\in\N$ and $\alpha>0$. For $r$ large enough, $\PP(\Omega(r,r^{\alpha})^{c})=\mathcal{O}(r^{-n})$.
\end{lemma}

\begin{proof}
Assume there exists $X$ in $(\R_{+})^{2}\cap D(O,r)$ such that $|X|-|A(X)|>r^{\alpha}$. Then, we can find a real number $\nu>0$ small enough and $z\in\Z^{2}$ satisfying $|\nu z - X|\leq\sqrt{2}\nu$ and $\textrm{Cyl}(\nu z,\rho/2)\subset\textrm{Cyl}(X,\rho)$. We can then deduce on the one hand that $|z|\leq 2r/\nu$ for $r$ large enough, and on the other hand, the existence of a deterministic set $\textrm{Cyl}(\nu z,\rho/2)^{\ast}$ included in $\textrm{Cyl}(X,\rho)^{\ast}$ (i.e. avoiding the PPP $\NN$)and whose area is larger than $\rho r^{\alpha}/4$. Hence we get
$$
\PP( \Omega(r,r^{\alpha})^{c} ) \leq (4r/\nu)^{2} e^{-\rho r^{\alpha}/4}
$$
from which Lemma \ref{lem:Omega} follows.
\end{proof}

For the rest of the proof, we choose real numbers $0<\alpha<\frac{1}{2}$, $0<\eps''<\eps'<\eps<\frac{1}{2}$ and $0<\varphi<\pi/2$. Let $C_{\varphi,\eps}$ be the following set:
$$
C_{\varphi,\eps} = \{ X \in \R^{2} , \; 0<\mbox{arg}(X)<\varphi \; \mbox{ and } \; |X|>r^{\frac{1}{2}+\eps} \} ~.
$$
On the event $\Omega(r,r^{\alpha})$ and for $r$ large enough, the couple $(X,A(X))$ satisfies condition (\ref{conditionAast}) whenever $X\in C_{\varphi,\eps}$. Indeed,
$$
|A(X)| - X(2) \geq |X| - r^{\alpha} - |X|\sin\varphi \geq r^{\frac{1}{2}+\eps}(1 - \sin\varphi) - r^{\alpha} \geq \rho
$$
for $r$ large enough. Hence, on the event $\Omega(r,r^{\alpha})$, we can define by induction a sequence $Z_{0},Z_{1}\ldots$ of points of $\R^{2}$ as follows. The starting point $Z_{0}$ satisfies $|Z_{0}|=r$, $Z_{0}(1)\geq 0$ and $Z_{0}(2)=r^{\frac{1}{2}+\eps'}$. While $Z_{k}$ belongs to $C_{\varphi,\eps}$, we set $Z_{k+1}=A^{\ast}(Z_{k})$.

Let us consider the random times
$$
\tau_{1} := \min \{ k\in\mathbb{N} \, , \; |Z_{k}| \leq r^{\frac{1}{2}+\eps} \}
$$
and
$$
\tau_{2} := \min \{ k\in\mathbb{N} \, , \; |Z_{k}(2)-Z_{0}(2)|>r^{\frac{1}{2}+\eps''} \} ~.
$$
Thus, we choose $r$ large enough so that $k<\tau_{1}\vee\tau_{2}$ implies that $Z_{k}$ belongs to $C_{\varphi,\eps}$: see Figure \ref{fig:straight}. Henceforth, the path $\mathbf{Z}=(Z_{0},\ldots,Z_{\tau_{1}\vee\tau_{2}})$ is well defined on $\Omega(r,r^{\alpha})$.

\begin{figure}[!ht]
\begin{center}
\psfrag{o}{\small{$O$}}
\psfrag{p}{\small{$\varphi$}}
\psfrag{z}{\small{$Z_{0}$}}
\psfrag{r1}{\small{$r^{\frac{1}{2}+\eps}$}}
\psfrag{r2}{\small{$r$}}
\psfrag{r3}{\small{$r^{\frac{1}{2}+\eps'}$}}
\psfrag{r4}{\small{$r^{\frac{1}{2}+\eps''}$}}
\includegraphics[width=14cm,height=6.5cm]{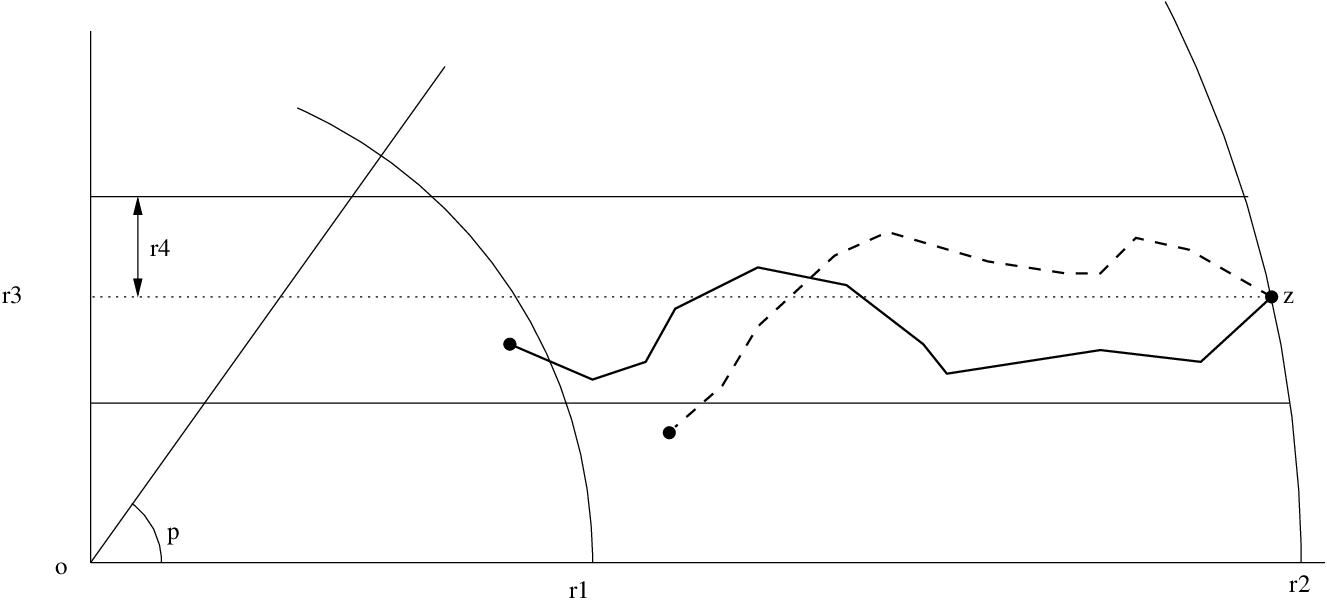}
\end{center}
\caption{\label{fig:straight} Here are two realizations of the path $\mathbf{Z}=(Z_{0},\ldots,Z_{\tau_{1}\vee\tau_{2}})$. The path in solid line satisfies $\tau_{1}<\tau_{2}$ whereas the path in dotted lines satisfies $\tau_{1}>\tau_{2}$.}
\end{figure}

\begin{lemma}
\label{lem:Zgras}
Assume the event $\Omega(r,r^{\alpha})$ satisfied. On the ring $D(O,r)\setminus D(O,|Z_{\tau_{1}\vee\tau_{2}}|)$, the path $\mathbf{Z}$ is above $\gamma^{+}_{r}$.
\end{lemma}

\begin{proof}
This result is based on the two following observations. Let $k<\tau_{1}\vee\tau_{2}$. First, the ancestors $A^{\ast}(Z_{k})$ and $A(Z_{k})$ are on the same circle with $A^{\ast}(Z_{k})(2)\geq A(Z_{k})(2)$. Second, the pathes of the RPT built on $\NN\cap(\R\!\times\!\R^{+})$ and starting at $A^{\ast}(Z_{k})$ and $A(Z_{k})$ do not cross.
\end{proof}

Let $n\in\mathbb{N}$ and assume that the maximal deviation $\Delta^{+}_{r}$ of $\gamma^{+}_{r}$ w.r.t. the horizontal axis $(OX_{0})$ is larger than $r^{\frac{1}{2}+\eps}$. Three cases can be distinguished:\\

Case 1: $\tau_{1}\leq\tau_{2}$. The path $\mathbf{Z}=(Z_{0},\ldots,Z_{\tau_{1}\vee\tau_{2}})$ enters in the disk $D(O,r^{\frac{1}{2}+\eps})$ before getting out the horizontal strip $\{X,|X(2)-Z_{0}(2)|\leq r^{\frac{1}{2}+\eps''}\}$. So, on the ring $D(O,r)\setminus D(O,r^{\frac{1}{2}+\eps})$, the path $\gamma^{+}_{r}$ is trapped between the axis $(OX_{0})$ and the path $\mathbf{Z}$ (Lemma \ref{lem:Zgras}). So, its maximal deviation $\Delta^{+}_{r}$ is smaller than $r^{\frac{1}{2}+\eps'}+r^{\frac{1}{2}+\eps''}$ which is smaller than $r^{\frac{1}{2}+\eps}$ for $r$ large enough. Moreover, once $\gamma^{+}_{r}$ is in $D(O,r^{\frac{1}{2}+\eps})$, it can no longer escape. Its maximal deviation inside this disk cannot exceed $r^{\frac{1}{2}+\eps}$. So Case 1 never happens.\\

Case 2: $\tau_{1}>\tau_{2}>\lfloor cr\rfloor$. In this case, the path $(Z_{0},\ldots,Z_{\lfloor cr\rfloor})$ is well defined but for $c$ large enough, it should had already entered in the disk $D(O,r^{\frac{1}{2}+\eps})$. Lemma \ref{lem:tau2}, proved at the end of the section, says that Case 2 occurs with small probability.

\begin{lemma}
\label{lem:tau2}
There exists a constant $c>0$ such that for any integer $n$ and $r$ large enough, the following statement holds:
$$
\PP( \tau_{1} > \tau_{2} > \lfloor cr\rfloor ,\, \Omega(r,r^{\alpha})) = \mathcal{O}(r^{-n}) ~.
$$
\end{lemma}

Case 3: $\tau_{1}>\tau_{2}$ and $\tau_{2}\leq\lfloor cr\rfloor$. Then, there exists an integer $m\in\{1,\ldots,\lfloor cr\rfloor\}$ such that the path $(Z_{0},\ldots,Z_{m})$ is well defined and satisfy $|Z_{k}(2)-Z_{0}(2)|\leq r^{\frac{1}{2}+\eps''}$ for any $k\in\{1,\ldots,m-1\}$, but
$$
|Z_{m}(2)-Z_{0}(2)| = \Big| \sum_{i=0}^{m-1} Z_{i+1}(2)-Z_{i}(2) \Big| > r^{\frac{1}{2}+\eps''} ~.
$$
Conditionally to $\tau_{1}>\tau_{2}$, $\tau_{2}=m$, $\Omega(r,r^{\alpha})$ and $|Z_{1}|,\ldots,|Z_{m}|$ we make two observations. First, by construction, the increments $(Z_{i+1}(2)-Z_{i}(2))_{0\leq i\leq m-1}$ are independent but not identically distributed (indeed, the law of $Z_{i+1}(2)-Z_{i}(2)$ depends on $|Z_{i}|$). Second, for any $i\in\{1,\ldots,m-1\}$, $Z_{i}(2)\geq r^{\frac{1}{2}+\eps'}-r^{\frac{1}{2}+\eps''}\geq r^{\alpha}$ for $r$ large enough. So, on $\Omega(r,r^{\alpha})$, the cylinder $\textrm{Cyl}(Z_{i},\rho)^{\ast}$ remains in the set $\{X, X(2)\geq 0\}$. By Lemma \ref{lem:starancestor}, this means that the increment $Z_{i+1}(2)-Z_{i}(2)$ is symmetrically distributed on $[-\rho;\rho]$: its conditional expectation is null. We can then apply Lemma \ref{lem:Youri} below to our context with $Y_{i+1}=Z_{i+1}(2)-Z_{i}(2)$ and $P$ given by the probability $\PP$ conditioned to $\tau_{1}>\tau_{2}$, $\tau_{2}=m$, $\Omega(r,r^{\alpha})$ and $|Z_{1}|,\ldots,|Z_{m}|$:
$$
\PP \left. \left( \begin{array}{c}
\forall k \in \{1,\ldots,m-1\} , \, |Z_{k}(2)-Z_{0}(2)| \leq r^{\frac{1}{2}+\eps''} \\
\mbox{ and } \; |Z_{m}(2)-Z_{0}(2)| > r^{\frac{1}{2}+\eps''}
\end{array} \; \right| \; \begin{array}{c}
\tau_{1}>\tau_{2}, \tau_{2}=m \\
\Omega(r,r^{\alpha}), |Z_{1}|,\ldots,|Z_{m}|
\end{array} \right) \hspace*{3cm}
$$
$$
\hspace*{2cm}= P \left( \forall k \in \{1,\ldots,m-1\} , \, \Big| \sum_{i=1}^{k} Y_{i} \Big| \leq r^{\frac{1}{2}+\eps''} \; \mbox{ and } \; \Big| \sum_{i=1}^{m} Y_{i} \Big| > r^{\frac{1}{2}+\eps''} \right)
$$
which is a $\mathcal{O}(r^{-n-1})$. Tacking the expectation, we obtain that the quantity $\PP(\tau_{1}>\tau_{2}, \tau_{2}=m, \Omega(r,r^{\alpha}))$ is a $\mathcal{O}(r^{-n-1})$, and then $\PP(\tau_{1}>\tau_{2}, \tau_{2}\leq\lfloor cr\rfloor, \Omega(r,r^{\alpha}))$ is a $\mathcal{O}(r^{-n})$.\\

In conclusion $\PP(\Delta_{r}^{+}\geq r^{\frac{1}{2}+\eps})$ is also a $\mathcal{O}(r^{-n})$ for all $\eps>0$. The same holds for $\PP(\Delta_{r}\geq r^{\frac{1}{2}+\eps})$, which achieves the proof of Proposition \ref{prop:RPTdeviation}.\\

The section ends with the proofs of Lemma \ref{lem:tau2}.

\begin{proof}
Let us assume that $\tau_{1}>\tau_{2}>\lfloor cr\rfloor$ for some positive constant $c$ which will be specified later, and $\Omega(r,r^{\alpha})$. Then, the path $(Z_{0},\ldots,Z_{\lfloor cr\rfloor})$ is well defined. Using $|Z_{i+1}|=|A^{\ast}(Z_{i})|=|A(Z_{i})|$, for any $0\leq i\leq\lfloor cr\rfloor-1$, we can write
\begin{equation}
\label{Z0<r}
|Z_{0}|-|Z_{\lfloor cr\rfloor}| = \sum_{i=0}^{\lfloor cr\rfloor-1} |Z_{i}|-|A(Z_{i})| \leq r ~.
\end{equation}
Now, we are going to define random variables $U_{i}$'s, for $0\leq i\leq\lfloor cr\rfloor-1$, which are identically distributed, and satisfy
\begin{equation}
\label{Ui}
\mbox{a.s. } \; |Z_{i}|-|A(Z_{i})| \geq U_{i} ~.
\end{equation}
To do it, let $\mathcal{C}_{i}$ be the set:
$$
\mathcal{C}_{i} = \Big( [O ; Z_{i}] \oplus D(O,\rho) \Big) \cap \{X\in\R^{2} , \langle X,Z_{i} \rangle < 1 \} ~.
$$
The right side of $\mathcal{C}_{i}$ is rectangular and so $\mathcal{C}_{i}$ contains the cylinder $\textrm{Cyl}(Z_{i},\rho)$. Thus, we denote by $H_{i}$ the element of $[O ; Z_{i}]$ with minimal norm such that $\mathcal{C}_{i}\setminus \left( [O ; H_{i}] \oplus D(O,\rho) \right)$ avoids the PPP $\NN\cap(\R\!\times\!\R^{+})$. In particular, $A(Z_{i})$ is on the circle $\CC(H_{i},\rho)$. Let us set $U_{i}=\max\{|Z_{i}-H_{i}|-\rho , 0\}$. By convexity, the disk $D(H_{i},\rho)$ is not include in $D(O,|A(Z_{i})|)$. This implies that $|Z_{i}|-|A(Z_{i})|$ is larger than $|Z_{i}-H_{i}|-\rho$. So (\ref{Ui}) is satisfied. Moreover, the distribution of $U_{i}$ does not depend on $Z_{i}$. This is due to the rectangular shape of $\mathcal{C}_{i}$ and, on the event $\Omega(r,r^{\alpha})$, $|Z_{i}|-|A(Z_{i})|$ is smaller than $r^{\alpha}$. In conclusion, the $U_{i}$'s are identically distributed.

Combining (\ref{Z0<r}) and (\ref{Ui}), it follows:
\begin{eqnarray*}
\PP( \tau_{1} > \tau_{2} > \lfloor cr\rfloor ,\, \Omega(r,r^{\alpha})) & \leq & \PP ( |Z_{0}|-|Z_{\lfloor cr\rfloor}| \leq r ,\, \Omega(r,r^{\alpha})) \\
& \leq & \PP \left( \sum_{i=0}^{\lfloor cr\rfloor-1} U_{i} \leq r ,\, \Omega(r,r^{\alpha}) \right) ~.
\end{eqnarray*}
By an abstract coupling, the $U_{i}$'s can be considered as independent random variables. Then, by Lemma \ref{lem:Gut}, $\PP(\sum_{i=0}^{\lfloor cr\rfloor-1} U_{i}\leq r , \Omega(r,r^{\alpha}))$ is a $\mathcal{O}(r^{-n})$ whenever $c>(\EE U_{1})^{-1}$. This leads to the searched result.

\end{proof}

\section{Technical lemmas}

This last section contains three technical results which are used many times in this paper. The first one is due to Talagrand (Lemma 11.1.1 of \cite{T}) and allows to bound from above the number of Poisson points occuring in a given set in terms of its area.

\begin{lemma}
\label{lem:Talagrand}
Let $\NN$ be a homogeneous PPP in $\R^{2}$ with intensity $1$. Then, for any bounded measurable set $\Lambda$ having a positive aera and any integer $n$,
$$
\PP \big( \NN (\Lambda) \ge n \big) \leq \exp \left( -n \ln\left( \frac{n}{e |\Lambda|} \right) \right) ~,
$$
where $|\Lambda|$ denotes the area of $\Lambda$.
\end{lemma}

The second result is a consequence of Theorem 3.1 of \cite{Gut}. It bounds the probability for a partial sum of i.i.d. random variables to be too small.

\begin{lemma}
\label{lem:Gut}
Let $(Y_{i})_{i\geq 1}$ be a sequence of positive i.i.d. random variables such that $\EE Y_{1}>0$ and for any integer $m$, $\EE Y_{1}^{m}<\infty$. Then, for any $m,r\geq 1$ and any constant $c>(\EE Y_{1})^{-1}$,
$$
\PP \left( \sum_{i=1}^{\lfloor cr\rfloor} Y_{i} \leq r \right) = \mathcal{O}(r^{-m}) ~.
$$
\end{lemma}

\begin{proof}
Since $c>(\EE Y_{1})^{-1}$, the inequality $\sum_{i=1}^{\lfloor cr\rfloor} Y_{i} \leq r$ implies that, for $r$ large enough, $\sum_{i=1}^{\lfloor cr\rfloor} X_{i} \leq -ar$ where $X_{i}=Y_{i}-\EE Y_{i}$ and $a>0$ is a constant. Henceforth,
$$
\PP \left( \sum_{i=1}^{\lfloor cr\rfloor} Y_{i} \leq r \right) \leq \PP \left( \left| \sum_{i=1}^{\lfloor cr\rfloor} X_{i} \right| \geq a r \right) = \mathcal{O}(r^{-m})
$$
by Theorem 3.1 of \cite{Gut} (precisely, equivalence between (3.1) and (3.2)) that we can use because $\EE X_{1}=0$ and for any $m$, $\EE |X_{1}|^{m}<\infty$.
\end{proof}

Lemma \ref{lem:Youri} says that with high probability, the maximal deviation of the first $t$ partial sums of a sequence $(Y_{i})_{i\geq 1}$ of independent bounded random variables (but not necessarily identically distributed) is smaller than $t^{\frac{1}{2}+\eps''}$.

\begin{lemma}
\label{lem:Youri}
Let $(Y_{i})_{i\geq 1}$ be a family of independent random variables defined on a probability space $(\Omega,\mathcal{F},P)$ satisfying for any $i\geq 1$, $E Y_{i}=0$ and $P(|Y_{i}|\leq\rho)=1$ (where $E$ denotes the expectation corresponding to $P$). Then, for all $0<\eps''<\frac{1}{2}$ and for all positive integers $n,t$:
$$
P \left( \max_{1\leq k\leq t} \Big| \sum_{i=1}^{k} Y_{i} \Big| \geq t^{\frac{1}{2}+\eps''} \right) = \mathcal{O}(t^{-n}) ~.
$$
Moreover, the $\mathcal{O}$ only depends on $\rho$, $n$ and $\eps''$.
\end{lemma}

\begin{proof}
Let $\eps''\in (0;\frac{1}{2})$, $n,t\in\mathbb{N}^{\ast}$ and $p=n/\eps''$. By independence of the $Y_{i}$'s, $(\sum_{i=0}^{k}Y_{i})_{k\geq 1}$ is a martingale. Applying the convex function $x\mapsto|x|^{p}$ ($p>1$), we get a positive submartingale $(|\sum_{i=0}^{k}Y_{i}|^{p})_{k\geq 1}$. Then, the Kolmogorov' submartingale inequality gives:
$$
P \left( \max_{1\leq k\leq t} \Big| \sum_{i=1}^{k} Y_{i} \Big| \geq t^{\frac{1}{2}+\eps''} \right) = P \left( \max_{1\leq k\leq t} \Big| \sum_{i=1}^{k} Y_{i} \Big|^{p} \geq t^{\frac{p}{2}+p\eps''} \right) \leq t^{-\frac{p}{2}-n} E \Big| \sum_{i=1}^{t} Y_{i} \Big|^{p} ~.
$$
So it remains to prove that $E|\sum_{i=1}^{t} Y_{i}|^{p}=\mathcal{O}(t^{\frac{p}{2}})$. At this time, bounding the $|Y_{i}|$'s by $\rho$ does not lead to the searched result. It is better to use Petrov \cite{Petrov} p.59:
$$
E\Big|\sum_{i=1}^{t} Y_{i}\Big|^{p} \leq C(p) E\Big(\sum_{i=1}^{t} Y_{i}^{2}\Big)^{\frac{p}{2}} \leq C(p) \rho^{p} t^{\frac{p}{2}}
$$
where $C(p)$ is a positive constant only depending on $p$.
\end{proof}

\section*{Aknowledgements}

The author thank the members of the ``Groupe de travail G\'eom\'etrie Stochastique'' of Universit\'e Lille 1 for enriching discussions, especially Y. Davydov who has indicated the inequality of \cite{Petrov} p.59 (used in Lemma \ref{lem:Youri}). This work was supported in part by the Labex CEMPI  (ANR-11-LABX-0007-01) and by the CNRS GdR 3477 GeoSto.

\small{\bibliographystyle{plain}
\bibliography{Big-bibli}}

\end{document}